\documentclass[11pt,oneside, letter]{article}

\usepackage{amsthm, amsmath, amssymb, amsfonts, graphicx, epsfig}
\usepackage{accents}

\usepackage{mathrsfs}

\usepackage{ulem}

\usepackage{cancel}

\usepackage{epstopdf}

\usepackage{graphicx}
\usepackage[font=sl,labelfont=bf]{caption}
\usepackage{subcaption}

\usepackage{float}
\restylefloat{table}

\usepackage{enumerate}

\usepackage[colorlinks=true, pdfstartview=FitV, linkcolor=blue,
            citecolor=blue, urlcolor=blue]{hyperref}
\usepackage[usenames]{color}

\usepackage{url}
\makeatletter\def\url@leostyle{%
 \@ifundefined{selectfont}{\def\UrlFont{\sf}}{\def\UrlFont{\scriptsize\ttfamily}}} \makeatother

\usepackage[margin=1.0in, letterpaper]{geometry}

\newtheorem{theorem}{Theorem}
\newtheorem{proposition}[theorem]{Proposition}
\newtheorem{lemma}[theorem]{Lemma}
\newtheorem{corollary}[theorem]{Corollary}

\theoremstyle{definition}

\newtheorem{example}[theorem]{Example}
\theoremstyle{remark}
\newtheorem{remark}[theorem]{Remark}

\numberwithin{equation}{section}
\numberwithin{theorem}{section}
\definecolor{Red}{rgb}{0.9,0,0.0}
\definecolor{Blue}{rgb}{0,0.0,1.0}
\def\cA{\mathcal{A}}
\def\cB{\mathcal{B}}
\def\cC{\mathcal{C}}

\def\cF{\mathcal{F}}

\def\cM{\mathcal{M}}

\def\bC{\mathbb{C}}

\def\bE{\mathbb{E}}

\def\bH{\mathbb{H}}

\def\bN{\mathbb{N}}

\def\bP{\mathbb{P}}
\def\bQ{\mathbb{Q}}
\def\bR{\mathbb{R}}

\def\bT{\mathbb{T}}

\newcommand{\1}{\mathbf{1}}            
\newcommand{\set}[1]{\{#1\}}            
\renewcommand{\d}{\operatorname{d}\!}   

\DeclareMathOperator*{\argmin}{arg\,min} 
\DeclareMathOperator*{\argmax}{arg\,max} 

\makeatletter
\def\namedlabel#1#2{\begingroup
    #2%
    \def\@currentlabel{#2}%
    \phantomsection\label{#1}\endgroup
}
\makeatother

\title{ 
 Long run risk sensitive portfolio with general factors}

\makeatletter
\def\and{%
  \end{tabular}%
  \begin{tabular}[t]{c}}%
\def\@fnsymbol#1{\ensuremath{\ifcase#1\or a\or b\or c\or
   d\or e\or f\or g\or h\or i\else\@ctrerr\fi}}
\makeatother

\author{
        Marcin Pitera\,\thanks{Institute of Mathematics, Jagiellonian University, Cracow, Poland,
       \newline \hspace*{1.45em}  Email: \url{marcin.pitera@im.uj.edu.pl}.
        \vspace{0.5em}}
\and and \ \
{\L}ukasz Stettner\,\thanks{
       Institute of Mathematics, Polish Academy of Sciences, Warsaw, Poland,
        \newline \hspace*{1.45em} Email: \url{l.stettner@impan.pl}, research supported by NCN grant DEC-2012/07/B/ST1/03298.
         \vspace{0.5em}}
        }

\date{ {\small%
 First circulated: August 21, 2015\\
 This version: \today}}

\begin{document}

\maketitle

{\footnotesize
\begin{tabular}{l@{} p{350pt}}
  \hline \\[-.2em]
  \textsc{Abstract}: \ & In the paper portfolio optimization over long run risk sensitive criterion is considered. It is assumed that economic factors which stimulate asset prices are ergodic but non necessarily uniformly ergodic. Solution to suitable Bellman equation using local span contraction with weighted norms is shown. The form of optimal strategy is presented and examples of market models satisfying imposed assumptions are shown. \\[0.5em]
\textsc{Keywords:} \ &  risk sensitive portfolio, Bellman equation, weighted span norms, risk measures \\
\textsc{MSC2010:} \ & 93E20, 91G10, 91G80 \\[1em]
  \hline
\end{tabular}
}

\bigskip


\section{Introduction}
Many stochastic control methods are used in theoretical studies of portfolio management (cf. \cite{Pri2007} and references therein). Among them, risk sensitive control is one of the most recognised ones. For infinite time horizon, any portfolio value process $V$ and risk-averse parameter $\gamma<0$, the Risk sensitive criterion (RSC) function is given by
\begin{equation}\label{eq:RSC.def}
\varphi^{\gamma}(V):=\liminf_{t\rightarrow\infty }\frac{1}{t}\frac{1}{\gamma}\ln E[V_{t}^{\gamma}].
\end{equation}

Using this objective function in portfolio management gives us many advantages over the standard theoretical methods, which are usually based on expected utility criterions. Let us alone mention difficulties associated with the estimation of model parameters or traceable difficulties which arise, when we try to compute optimal trading strategies for the realistic security market models \cite{BiePli2003}. For RSC, applying Taylor expansion around $\gamma=0$, we get
\begin{equation}\label{eq:RSC:rtorew}
\varphi^{\gamma}(V)=\liminf_{t\rightarrow\infty }\frac{1}{t}\Big[E[\ln V_{t}]+\frac{\gamma}{2}Var(\ln V_{t}) +O(\gamma^{2},t)\Big],
\end{equation}
which shows that this map could be seen as a measure of performance, as it penalise expected growth rate with asymptotic variance multiplied by risk-averse parameter $\gamma<0$. Of course, this only applies for problems, for which the last term (i.e. $O(\gamma^{2},t)/t$) vanishes, when $t$ goes to infinity. Nevertheless, this assumption is satisfied for a lot of standard dynamics, as explained in~\cite[Section 5]{BiePli2003}, so \eqref{eq:RSC:rtorew} brings out the motivation, which led to this class of maps. We refer to \cite{BiePli2003} for a further discussion about economic properties of RSC.

Following~\cite{BieCiaPit2013,GulRus2015}, we would like to stress out the fact, that RSC could be seen as a risk-to-reward criterion. In fact, RSC could be considered as an {\it Acceptability index} \cite{CheMad2009,BieCiaZha2012}, the map quantifying the tradeoff between portfolio growth and the risk associated with it. Many methods from risk and performance measurement theory could be directly applied to RSC, as we will show in this paper.

From another point of view, RSC is a good objective function for many optimal control problems related to (controlled) Markov decision processes both on finite and infinite time horizons (cf. \cite{HerLas1996,Her1989,DiMSte1999,CavHer2005} and references therein). In particular, the connection to portfolio optimization was shown in \cite{BiePli1999}, where RSC was applied to continuous time infinite time horizon, and a version of Merton's intertemporal capital asset pricing model \cite{Mer1973} was considered. The analogous study for discrete time market model was done in \cite{Ste1999}.

Because of that, we have decided to present our results in such a way, that they might be interesting both for specialists from risk analysis, in particular studying dynamic growth indices, as well as for specialists from risk sensitive control Markov decision processes.

There are many sophisticated methods, which guarantee the existence of the solution to Bellman equation associated with RSC. Let us alone mention the {\it vanishing discount approach}~\cite{HerMar1996} or the {\it fixed point approach}~\cite{DiMSte1999}. The assumptions under which the existence of the solutions is guaranteed are usually related to ergodic properties of the considered process~\cite{DiMSte1999,KonMey2003,Her1989,HerMar1996}. The most recent results relate to localized Doeblin's conditions~\cite{CavHer2005} and Markov splitting techniques \cite{DiMSte2006}. The theory of RSC is also strictly connected to multiplicative Poisson equations \cite{DiMSte2006} and Issacs equations for ergodic cost stochastic dynamic games (cf. \cite{HerMar1996,FleHer1997,PraMenRun1996} and references therein).

In the paper, we generalize the results of \cite{Ste1999} in the sense that we consider market model with more general economic factors, which are not necessarily uniformly ergodic, and consequently studying Bellman equation we have to work with suitable weight functions.  Such more general economic factors were studied for Black Scholes market in the paper \cite{BiePli1999} and then continued for continuous time general diffusion models in \cite{Nag2003}. In this paper we are studying discrete time model and we were motivated by attempts to generalize risk neutral results of \cite{HaiMat2011} to the risk sensitive portfolio by the paper \cite{SheStaObe2013}.

The main novelty of the paper is that we obtain, using weighted span norm contraction method, the existence of solutions to suitable Bellman equation. Consequently, our paper can be applied to more general dynamics of the market than in \cite{Ste1999}. Furthermore we solve a risk sensitive control problem with unbounded solutions to the Bellman equation.

This paper is organized in follows. Section~\ref{S:Preliminaries} is the general setup. We state here all assumptions core to our study (e.g. on dynamics, control, etc.).  Next, in Section~\ref{S:norms} we recall some basic notation for the weighted norms and span-norms. In Section~\ref{S:RSC.Bellman} we present the main results of this paper, i.e. we state the Bellman equation and show when it could be solved. In Section~\ref{S:Strategy} we show how to connect Bellman equation to the initial investment problem. In particular we discuss how, given a solution to Bellman equation, construct the optimal strategy and when it is possible. Finally, in Section~\ref{S:RSC.Examples} we show exemplary dynamics, that could be fit to our model.

\section{Preliminaries}\label{S:Preliminaries}
Let $(\Omega,\cF,\{\cF_{t}\}_{t\in\bT},\bP)$ be a discrete-time filtered probability space, where $\bT=\bN$, $\cF_{0}$ is trivial and $\cF=\bigcup_{t\in\bT}\cF_{t}\,$. Moreover, let $L^{0}:=L^{0}(\Omega,\cF,\bP)$ denote the space of all (a.s. identified) $\cF$-measurable random variables.

We will assume that the market consists of $m$ risky assets (e.g. stocks, bonds, derivative securities) and $k$ economical factors (e.g. rates of inflation, short term interest rates, dividend yields). Prices of $m$ risky assets will be denoted by $S^{i}=(S_{t}^{i})_{t\in\bT}$ for ($i=1,\ldots,m$) and levels of $k$ economical factors will be denoted by $X^{j}=(X_{t}^{j})_{t\in\bT}$ for ($j=1,\ldots,k$). For simplicity, we will write $S:=(S_{t})_{t\in\bT}$ and $X:=(X_{t})_{t\in\bT}$, where $S_{t}=(S_{t}^{1},\ldots,S_{t}^{m})$ and $X_{t}=(X_{t}^{1},\ldots,X_{t}^{k})$.

We will use $\cA$ to denote the set of all $U$-valued adapted processes, where $U$ is a compact subset of $\bR^{m}$. Elements of $\cA$ will
correspond to all admissible portfolio strategies $H:=(H_{t})_{t\in\bT}$, where $H_{t}=(H_{t}^1,\ldots,H_{t}^m)$ and $H^{i}=(H^i_t)_{t\in\bT}$ is a part of capital invested in $i$-th risky asset (for $i=1,\ldots,m$). Furthermore, we will use notation $V^{H}=(V_{t}^{H})_{t\in\bT}$ to denote the portfolio value process corresponding to strategy $H$.

Throughout this paper we will make the following assumptions:

\begin{enumerate}

\item[(\namedlabel{as:rsc:A.1}{A.1})] The filtration $\{\cF_{t}\}_{t\in\bT}$ will be generated by a sequence of $k+m$ stochastic processes denoted by $W^{i}=(W^{i}_{t})_{t\in\bT}$ for $(i=1,\ldots,k+m)$. Moreover, $W_{t}=(W_{t}^{1},\ldots,W_{t}^{k+m})$ will be independent of $\cF_{t}$ and $Law(W_{t+1})=Law(W_{t})$, i.e. $W:=(W_{t})_{t\in\bT}$ will form a sequence of i.i.d. random vectors.

\item[(\namedlabel{as:rsc:A.2}{A.2})] The factor process $X$ will be Markov and will admit the following representation:
\[
X_0\in\bR^{k},\quad X_{t+1}=G(X_{t},W_{t}):=(G^{1}(X_{t},W_{t}),\ldots, G^{k}(X_{t},W_{t})),
\]
where $G^{i}\colon \bR^{k} \times \bR^{k+m}\to \bR^{k}$ is a Borel measurable function, continuous with respect to the first variable (for $i=1,\ldots,k$).
\item[(\namedlabel{as:rsc:A.3}{A.3})] For any $H\in\cA$, we will assume that the portfolio dynamics will be of the form
\begin{equation}\label{eq:rsc.dynamics}
V^{H}_{0}=V_{0},\quad\quad \ln\frac{V^{H}_{t+1}}{V^{H}_{t}}=F(X_t,H_t,W_t),
\end{equation}
for $t\in \bT$, where $V_{0}>0$ and $F\colon\bR^k\times U\times \bR^{k+m}\to \bR$ is a Borel measurable function, continuous with respect to the first two variables.

\item[(\namedlabel{as:rsc:A.4}{A.4})]  We will assume that for any $t\in\bT$, $x\in\bR^{k}$, $h\in U$ we have
\begin{align}
\omega(G(x,w)) &\leq a_{1}(w)+b_{1}\omega(x)\label{eq:assumpt.rsc.G},\\
|F(x,h,w)| & \leq a_{2}(w)+b_{2}\omega(x) \label{eq:assumpt.rsc},
\end{align}
for Borel measurable functions $a_1,a_2\colon \bR^{k+m}\to\bR_{+}$, constants $b_1\in (0,1)$, $b_2>0$ and continuous measurable function
$\omega\colon \bR^{k}\to [0,\infty)$, which we shall refer to as {\it the weight function}.
Moreover, we will assume that for any $\gamma\in \bR$,
\begin{equation}\label{eq:assumpt.kpm}
\mu^{\gamma}(a_1(W_{0}))\in\bR\quad\textrm{and}\quad\mu^{\gamma}(a_2(W_{0}))\in\bR,
\end{equation}
where $\mu^{\gamma}:L^{0}\to\bar{\bR}$ is the entropic utility measure, i.e.
\begin{equation}\label{eq:entrDRM}
\mu^{\gamma}(X):=\left\{
\begin{array}{ll}
\frac{1}{\gamma}\ln \bE[\exp(\gamma X)] &\quad \textrm{if } \gamma\neq 0,\\
\bE[X] & \quad \textrm{if } \gamma=0.
\end{array}\right.
\end{equation}

\item[(\namedlabel{as:rsc:A.5}{A.5})] For any $R>0$, there exists a constant $c>0$ and probability measure $\nu$, such that
\begin{equation}\label{eq:rsc:unerg}
\inf_{x\in C_{R}}\bP[G(x,W_0)\in A]\geq c\nu(A),\quad A\in\cB(\bR^{k}),
\end{equation}
where $C_{R}=\{x\in\bR^{k}\colon \omega(x)\leq R\}$.
\end{enumerate}

Assumptions \eqref{as:rsc:A.1} and \eqref{as:rsc:A.2} are classic conditions imposed on the probability space and the factor process, respectively.

Assumption \eqref{as:rsc:A.3} is technical -- it allows to model portfolios through log-returns, rather than value processes (see e.g.  Example~\ref{ex:rsc1} or \cite{Ste1999} for more details).

Assumption \eqref{as:rsc:A.4} has a financial interpretation. The state-space constraints $b_1$ and $b_2$ introduced in  \eqref{eq:assumpt.rsc.G} and \eqref{eq:assumpt.rsc} say that in our model we allow only $\omega$-growth (i.e. growth proportional to the growth of $\omega$) with respect to the state space. In particular, inequality \eqref{eq:assumpt.rsc.G} might be seen as a form of the geometric drift condition imposed on $X$  (cf. \cite{HaiMat2011}).
On the other hand, assumption \eqref{eq:assumpt.kpm} allow us to have control over the entropy of the noise part. In a more probabilistic setting, it is equivalent to the statement that the moment generating functions for $a_1(W_{0})$ and $a_2(W_{0})$ exist. In particular, we might say that the utility (or risk) of a single period log-return at time $t$ measured by $\mu^{\gamma}$ (or $-\mu^{\gamma}$) must be finite for any simple trade (in any fixed state) and in fact it is bounded by $\pm a_2(W_{t})$ plus some constant (dependant on the state). Please note, that this assumption is rather weak, and fulfilled by standard models, which describe log-returns as processes of the form
\[
F(x,h,W_{t})=a(x,h,W_{t})+\sum_{i=1}^{k+m}b(x,h)W_{t}^{i},
\]
where $W_{t}$ is a random vector with multidimensional normal distribution and functions $a$ and $b$ satisfy $\omega$-growth constraints. Then, the function $a_2$ could be constructed using random variables $\min(W_{t}^{1},\ldots,W_{t}^{k+m})$ and $\max(W_{t}^{1},\ldots,W_{t}^{k+m})$.

Assumption \eqref{as:rsc:A.5} is a (local) minorization property. Combined with the geometric drift condition, it allow us to exploit the ergodic properties of $X$ (cf. \cite{HaiMat2011}). Please note that setting $\omega\equiv 0$, for any $R>0$ we get $C=\bR^{k}$. Consequently, in this particular case, \eqref{as:rsc:A.5} becomes a global Doeblin's condition, which is equivalent to the uniform ergodicity of process $X$. On the other hand, if $\omega$ is unbounded and $C_{R}$ is compact for any $R>0$, then \eqref{eq:rsc:unerg} is directly linked to the (local) mixing condition, i.e. the statement that for any fixed compact subset $K$ (of $\bR^{k}$), we get
\begin{equation}\label{eq:loc.min}
\sup_{x,y\in K}\sup_{A\in\cB(\bR^{k})}\left|\bP[G(x,W_1)\in A]-\bP[G(y,W_1)\in A]\right|<1.
\end{equation}

The main goal of this paper is to optimize the risk sensitive cost criterion $\varphi^{\gamma}$ given by~\eqref{eq:RSC.def}, i.e.
\[
\varphi^{\gamma}(V)=\liminf_{t\rightarrow\infty }\frac{1}{t}\frac{1}{\gamma}\ln E[V_{t}^{\gamma}],
\]
where $\gamma<0$ is a fixed risk aversion parameter and $V$ is portfolio value process. In other words, given the set $\cA$ and dynamics of $V^{H}$ for any $H\in\cA$, we want to solve the optimal stochastic control problem
\begin{equation}\label{eq:RSC:begin}
\sup_{H\in\cA}\varphi^{\gamma}(V^{H}).
\end{equation}
Using the entropic representation of $\varphi^{\gamma}$~(see \cite{BieCiaPit2013} for more details) and \eqref{eq:rsc.dynamics}, for any $H\in\cA$, we get
\begin{equation}\label{eq:rsc.1}
\varphi^{\gamma}(V^{H})=\liminf_{t\to\infty}\frac{\mu^{\gamma}\big(\ln
\frac{V^{H}_{t}}{V^{H}_{0}}\big)}{t}=\liminf_{t\to\infty}\frac{\mu^{\gamma}(\sum_{i=0}^{t-1}F(X_i,H_i,W_i))}{t},
\end{equation}
where $\mu^{\gamma}$ is entropic utility measure given by \eqref{eq:entrDRM}.
Note that the first equality in~\eqref{eq:rsc.1} provides another financial interpretation of the RSC. The logarithmic transform of $V_{t}^{H}$ allow us to measure the cumulative growth (log return) at time $t$, while the map $\mu^{\gamma}$ is used to evaluate its (entropic) utility. Then, we divide the outcome by $t$ to normalise it in time and use $\liminf$ to measure (a worst case robust version of) the long-time efficiency of the value process (cf. \cite{BieCiaPit2013}).

Under the above assumptions, from \eqref{eq:rsc.1}, it is not difficult to see, that the optimal value of the problem \eqref{eq:RSC:begin} will be finite, which is in fact the statement of Proposition~\ref{pr:RSC.fin}.
\begin{proposition}\label{pr:RSC.fin}
Let $\gamma<0$. Under assumptions \eqref{as:rsc:A.1}--\eqref{as:rsc:A.4}, we get
\[-\infty<\sup_{H\in\cA}\varphi^{\gamma}(V^{H})<\infty.\]
\end{proposition}

\begin{proof}\textbf{}
Using \eqref{as:rsc:A.3} and \eqref{as:rsc:A.4}, for any $H\in\cA$ and $t\in\bT$, we get
\begin{align*}
\sum_{i=0}^{t-1}F(X_i,H_i,W_i) & \leq \sum_{i=0}^{t-1}a_2(W_{i})+b_2\omega(X_{i})\\
& \leq \sum_{i=0}^{t-1}\left(a_2(W_{i}) +b_2\left(b_1^{i}\omega(X_{0})+\sum_{j=0}^{i-1}b_1^{j}a_1(W_{i-j})\right)\right)\\
& \leq \frac{b_2}{1-b_1}\omega(X_{0})+ \sum_{i=0}^{t-1}\left( a_2(W_{i})+\frac{b_2}{1-b_1} a_1(W_{i})\right).
\end{align*}

As the entropic utility measure $\mu^{\gamma}$ is monotone, translation invariant, additive for any two independent random variables and law
invariant
\cite{KupSch2009}, for any $t\in\bT$, we get
\begin{align*}
\mu^{\gamma}\left(\sum_{i=0}^{t-1}F(X_i,H_i,W_i)\right) & \leq
\frac{b_2}{1-b_1}\omega(X_{0})+\sum_{i=0}^{t-1}\mu^{\gamma}\left(a_2(W_{i})+\frac{b_2}{1-b_1}a_1(W_{i})\right)\\
& =\frac{b_2}{1-b_1}\omega(X_{0})+t\mu^{\gamma}\left(a_2(W_{0})+\frac{b_2}{1-b_1}a_1(W_{0})\right).
\end{align*}
Consequently, using \eqref{eq:rsc.1} and \eqref{eq:assumpt.kpm}, for any $H\in\cA$, we get
\[
\varphi^{\gamma}(V^{H})= \liminf_{t\rightarrow\infty }\frac{\mu^{\gamma}\left(\sum_{i=0}^{t-1}F(X_i,H_i,W_i)\right)}{t}\leq
\mu^{\gamma}\left(a_2(W_{0})+\frac{b_2}{1-b_1}a_1(W_{0})\right)<\infty.
\]
The proof of the other inequality is analogous.
\end{proof}

\section{Weighted norms}\label{S:norms}
In assumption \eqref{as:rsc:A.4} we have introduced measurable and continuous function  $\omega\colon \bR^{k}\to [0,\infty)$, which we referred
to as
{\it the weight function}. Following~\cite{HaiMat2011} let us now recall basic notation regarding those function. We shall denote by
$\cC_{\omega}(\bR^{k})$ the set of all continuous and measurable functions $f:\bR^{k}\to\bR$, such that the $\omega$-norm of $f$
is
bounded, i.e.
\[
\|f\|_{\omega}:=\sup_{x\in\bR^{k}}\frac{|f(x)|}{1+\omega(x)}<\infty.
\]
Next, we define $\omega$-span seminorm of $f\in \cC_{\omega}(\bR^{k})$ by
\[
\|f\|_{\omega\textrm{-span}}:=\sup_{x,y\in\bR^{k}}\frac{f(x)-f(y)}{2+\omega(x)+\omega(y)}.
\]

\begin{remark}\label{rem:span}
The classic span-norm of function $f\colon\bR^{k}\to \bR$ (cf.~\cite{HerLas1996} and references therein) is usually defined as $\| f \|_{\textrm{span}}=\sup_{x}f(x)-\inf_{y}f(y)$. Note that in our framework, using $\omega\equiv 0$, we get $\|f\|_{\omega\textrm{-span}}=\frac{\sup_{x}f(x)-\inf_{x}f(x)}{2}=\frac{1}{2}\|f\|_{\textrm{span}}$. Moreover, for any bounded weight function $\omega$, we know that $\| \cdot \|_{\textrm{span}}$ and $\|\cdot\|_{\omega\textrm{-span}}$ are equivalent.
\end{remark}

For any $\beta>0$ we shall also define the weighted (semi)norms given by
 \begin{align*}
\|f\|_{\beta,\omega} & :=\sup_{x\in\bR^{k}}\frac{|f(x)|}{1+\beta\omega(x)},\\
\|f\|_{\beta,\omega\textrm{-span}} &:=\sup_{x,y\in\bR^{k}}\frac{f(x)-f(y)}{2+\beta\omega(x)+\beta\omega(y)}.
\end{align*}
Please note that for any $\beta>0$ and $c\geq 0$, the function $\omega'\colon \bR^{k}\to [0,\infty)$, given by $\omega'(x)=\beta\omega(x)+c$ is
also
a weight function. Let us now recall some basic properties of weighted norms and related span norms.

\begin{proposition}\label{pr:omega}
Let $\omega\colon \bR^{k}\to [0,\infty)$ be a weight function. Then
\begin{enumerate}[1)]
\item For any $\beta>0$, the norms $\|\cdot\|_{\omega}$ and $\|\cdot\|_{\beta,\omega}$ are equivalent.
\item For any $\beta>0$, the seminorms $\|\cdot\|_{\omega\textrm{-span}}$ and $\|\cdot\|_{\beta,\omega\textrm{-span}}$ are equivalent.
\item For any $0<\beta<1$ and $f\in\cC_{\omega}(\bR^{k})$, we get $\|f\|_{\omega\textrm{-span}}\leq \|f\|_{\beta,\omega\textrm{-span}}$.
\item For any $f\in\cC_{\omega}(\bR^{k})$ we get $\inf_{c\in\bR}\|f+c\|_{\omega}=\|f\|_{\omega\textrm{-span}}$.
\item Let $f\in \cC_{\omega}(\bR^{k})$ and $c\in\bR$. Then $\|f+c\|_{\omega}=\|f\|_{\omega\textrm{-span}}$ if and only if $c\in [c_1,c_2]$, where
\begin{equation}\label{eq:c.minus}
c_1=-\inf_{x\in\bR^{k}} \left\{f(x)+(1+\omega(x))\|f\|_{\omega\textrm{-span}}\right\},
\end{equation}
\begin{equation}\label{eq:c.plus}
c_2=-\sup_{x\in\bR^{k}} \left\{f(x)-(1+\omega(x))\|f\|_{\omega\textrm{-span}}\right\}.
\end{equation}
Moreover,  there exists $c_0\in \{c_1,c_2\}$, such that
\begin{equation}\label{eq:centering.span}
\|f+c_0\|_{\omega}=\sup_{x\in\bR^{k}}\frac{f(x)+c_0}{1+\omega(x)}=-\inf_{x\in\bR^{k}}\frac{f(x)+c_0}{1+\omega(x)}.
\end{equation}
\end{enumerate}
\end{proposition}
\begin{proof}
The proof of properties 1), 2) and 3) is straightforward and hence omitted here.

4) The proof is based on \cite[Lemma 2.1]{HaiMat2011} and is recalled for completeness. Let  $f\in \cC_{\omega}(\bR^{k})$.

For any $x\in \bR^{k}$, we get $|f(x)|\leq \|f\|_{\omega}(1+\omega(x))$, which in turn implies
\[
{f(x)-f(y)\over 2+ \omega(x)+\omega(y)}\leq {\|f\|_{\omega}\left[2+\omega(x)+\omega(y)\right] \over 2+\omega(x) + \omega(y)}=\|f\|_{\omega},\quad\quad x,y\in\bR^{k}.
\]
Consequently, for any $c\in\bR$ we get
\begin{equation}\label{eq:inequality.a1}
\|f\|_{\omega\textrm{-span}}=\|f+c\|_{\omega\textrm{-span}}\leq \|f+c\|_{\omega}.
\end{equation}
Let us now prove the other inequality. Noting, that we could take $a\cdot f$ instead of $f$, for some $a>0$ and the proof for the case $\|f\|_{\omega\textrm{-span}}=0$ is trivial, without loss of generality we could assume that $\|f\|_{\omega\textrm{-span}}=1$. By the definition of $\|\cdot\|_{\omega\textrm{-span}}$ and the fact that  $\|f\|_{\omega\textrm{-span}}=1$, we get
\[
f(x)- [f(y)+1 +\omega(y)]\leq 1+\omega(x),
\]
for any $x,y \in \bR^{k}$. Thus, $c_1:=-\inf_{y\in\bR^{k}} \left\{f(y)+1+\omega(y)\right\}\in\bR$ and for any $x\in\bR^k$, we get
\begin{equation}\label{eq:a1}
f(x)+c_1= \sup_{y\in\bR^k} \left[f(x)-f(y)-1-\omega(y)\right]\leq 1+\omega(x).
\end{equation}
On the other hand, for any $x\in\bR^{k}$, we get
\begin{equation}\label{eq:a2}
f(x)+c_1=\sup_{y\in\bR^k} \left[f(x)-f(y)-1-\omega(y)\right]\geq f(x)-f(x)-1-\omega(x)= -(1+\omega(x)).
\end{equation}
Combining \eqref{eq:a1} and \eqref{eq:a2}, we get $\|f+c_1\|_{\omega}\leq 1$. This, together with \eqref{eq:inequality.a1}, concludes the proof of 4).

5) Let  $f\in \cC_{\omega}(\bR^{k})$ and let $c\in\bR$. Repeating and slightly modifying the proof of 4) it is easy to check that
\begin{equation}\label{eq:5.1}
\|f+c_1\|_{\omega}=\|f+c_2\|_{\omega}=\|f\|_{\omega\textrm{-span}}.
\end{equation}
If $c\in[c_1,c_2]$, then there exists $\alpha\in [0,1]$ such that $c=\alpha c_1+(1-\alpha) c_2$. Thus, using \eqref{eq:inequality.a1} and \eqref{eq:5.1}, we get
\[
\|f\|_{\omega\textrm{-span}}\leq \|f+c\|_{\omega} \leq \alpha \|f+c_1\|_{\omega} + (1-\alpha)\|f+c_2\|_{\omega}=\|f\|_{\omega\textrm{-span}}.
\]
On the other hand, we know that if $\|f+c\|_{\omega}=\|f\|_{\omega\textrm{-span}}$, then for any $x\in\bR^{k}$ we get
\[
-\|f\|_{\omega\textrm{-span}}\leq \frac{f(x)+c}{1+\omega(x)}\leq \|f\|_{\omega\textrm{-span}}.
\]
Because of that, for any $x\in\bR^{k}$ we have
\[
-f(x)-(1+\omega(x))\|f\|_{\omega\textrm{-span}}\leq c \leq -f(x)+(1+\omega(x))\|f\|_{\omega\textrm{-span}},
\]
and consequently $c_1\leq c\leq c_2$. This completes the first part of the proof.
Let us now show that there exists (at least one) $c_0\in [c_1,c_2]$, satisfying \eqref{eq:centering.span}.

Given $f\in \cC_{\omega}(\bR^{k})$, for any $c\in\bR$ we define
\[
a_{+}(c):=\sup_{z\in\bR^{k}}\frac{f(z)+c}{1+\omega(z)}\quad\textrm{and}\quad
a_{-}(c):=-\inf_{z\in\bR^{k}}\frac{f(z)+c}{1+\omega(z)}.
\]
It is easy to note that $a_+(\cdot)$ is finite, continuous and non-decreasing, while $a_-(\cdot)$ is finite, continuous and non-increasing. Moreover
$a_+(c)\to\infty$, as $c\to\infty$, and $a_-(c)\to\infty$, as $c\to-\infty$.
Thus, there exists $c_0\in\bR$, such that $a_{+}(c_0)=a_{-}(c_0)$. Moreover, for any $c\geq c_0$ we get
\[
\|f+c\|_{\omega}=\max(a_{+}(c),a_{-}(c))\geq a_{+}(c_0)=\max(a_{+}(c_0),a_{-}(c_0))=\|f+c_0\|_{\omega},
\]
while for  $c\leq c_0$ w get
\[
\|f+c\|_{\omega}=\max(a_{+}(c),a_{-}(c))\geq a_{-}(c_0)=\max(a_{+}(c_0),a_{-}(c_0))=\|f+c_0\|_{\omega}.
\]
Consequently,
\begin{equation}\label{eq:w.span.eq}
a_{+}(c_0)=a_{-}(c_0)=\|f+c_0\|_{\omega}=\inf_{c\in\bR}\|f+c\|_{\omega}=\|f\|_{\omega\textrm{-span}}.
\end{equation}
By the first part of the proof of 5), we know that $c_0\in [c_1,c_2]$.
If $c_0$ is equal to $c_1$ or $c_2$, then the proof is finished. On the contrary, let us assume that $c_0\not\in\{c_1,c_2\}$. By 
using monotonicity of $a_{+}(\cdot)$ we have $a_{+}(c_0)\leq a_{+}(c_2)$  and by \eqref{eq:w.span.eq} using 
\[
\|f+c_0\|_{\omega}=\|f+c_1\|_{\omega}=\|f+c_2\|_{\omega}=\max(a_{+}(c_2),a_{-}(c_2)),
\]
we obtain $ a_{+}(c_2)= a_{+}(c_0)$. Consequently $a_+(\cdot)$ must be constant on $[c_0,c_2]$ and as a convex nondecreasing mapping it is in fact constant on $(-\infty,c_2]$.
Using similar arguments, we get that $a_-(\cdot)$  as a nonincreasing convex mapping must be constant on $[c_1,\infty]$. Consequently, both $c_1$ and $c_2$ satisfy \eqref{eq:centering.span}, which concludes the proof.
\end{proof}


\begin{remark}
We might get $c_1\neq c_2$. Let $f(x)=0$ for $|x|\leq 1$, and $f(x)=|x-{1\over x}|$ for $|x|\geq 1$. Then, for $\omega(x)=|x|$, it is easy to check that $\|f\|_{\omega\textrm{-span}}=1$, $c_1=-1$ and $c_2=1$.
Moreover, one might look at $c_0$ as a centering constant for weighted $f$, i.e. the constant, such that the distance from $0$ to $\sup_{x\in\bR^{k}}\frac{f(x)+c_0}{1+\omega(x)}$ is the same as the distance from 0 to $\inf_{x\in\bR^{k}}\frac{f(x)+c_0}{1+\omega(x)}$. In particular, the $\|\cdot\|_{\omega\textrm{-span}}$ seminorm might be considered as a $\|\cdot\|_{\omega}$ norm for centered function, which provide some insight for 4) in Proposition~\ref{pr:omega}.
\end{remark}

Proposition~\ref{pr:omega} implies that for any $\beta>0$, $c\geq 0$, $f:\bR^{k}\to \bR$ and $\omega'$ defined by $\omega'(x)=\beta\omega(x)+c$,
we
get
\begin{equation}\label{eq:omega'}
\|f\|_{\omega}<\infty \iff \|f\|_{\omega'}<\infty,
\end{equation}
which in turn implies
\[
\cC_{\omega}(\bR^{k})=\cC_{\omega'}(\bR^{k}).
\]
Moreover, if a family of functions is uniformly bounded wrt. $\omega$-span norm, then it is uniformly bounded wrt. $\omega'$-span norm.

Next, for any $\beta>0$, two probability measures $\bQ_{1}$ and $\bQ_{2}$ on $(\bR^{k},\cB(\bR^{k}))$ and the corresponding signed measure $\bH=\bQ_{1}-\bQ_{2}$, let $\|\bH\|_{\beta,\omega\textrm{-var}}$ denote its weighted total variation norm given by
\[
\|\bH \|_{\beta,\omega\textrm{-var}}
=\int_{\bR^{k}}\big(1+\beta\omega(z)\big)|\bH|(dz)=\sup_{\varphi: \|\varphi\|_{\beta,\omega}\leq 1}\int_{\bR^{k}}\varphi(z)\bH(dz),
\]
where $|\bH|$ denote the total variation of $\bH$, i.e.
\[
|\bH|=1_A \bH - 1_{A^c}\bH,
\]
for $A$ being a positive set for measure $\bH$ (obtained e.g. using Hahn-Jordan decomposition). In particular (for $\omega\equiv 0$), let $\|\bH\|_{\textrm{var}}$ denote the the standard total variation norm \cite{HerLas1996}, i.e.
\[
\|\bH \|_{\textrm{var}}:=\int_{\bR^{k}}|\bH|(dz)=2\sup_{A\in\cB(\bR^{k})}|\bQ_{1}(A)-\bQ_{2}(A)|.
\]

\section{Bellman equation}\label{S:RSC.Bellman}
Using representation \eqref{eq:rsc.1}, it is not hard to see that the Bellman equation corresponding to \eqref{eq:RSC:begin} is of the form
\begin{equation}\label{eq:rsc:bellmaneq}
v(x)+\lambda=\sup_{h\in U}\mu^{\gamma}(F(x,h,W_0)+v(G(x,W_0))),
\end{equation}
where $\lambda\in\bR$, $v\in \cC_{\omega}(\bR^{k})$, $x\in\bR^{k}$ and $\omega\colon \bR^{k}\to [0,\infty)$ is a weight function from
\eqref{as:rsc:A.4}, for which the corresponding Bellman operator
\begin{equation}\label{Rf}
R_{\gamma}f(x):=\sup_{h\in U}\mu^{\gamma}(F(x,h,W_0)+f(G(x,W_0))),\quad f\in \cC_{\omega}(\bR^{k}),
\end{equation}
satisfies certain contraction properties.

For computational convenience, let us introduce the associated Bellman equation
\begin{align}
u(x)+\lambda \gamma &=\gamma\sup_{h\in U}\mu^{\gamma}(F(x,h,W_0)+{u(G(x,W_0))\over \gamma})\nonumber\\
& =\inf_{h\in U}\ln \bE[e^{\gamma F(x,h,W_0)+u(G(x,W_0))}]\nonumber\\
& =T_{\gamma}u(x) \label{eq:rsc:bellmaneq2},
\end{align}
where $u(x)=\gamma v(x)$ and where the corresponding Bellman operator takes the form
\begin{equation}\label{Tf}
T_{\gamma}f(x):=\gamma R_{\gamma}\frac{f(x)}{\gamma}=\inf_{h\in U}\ln \bE[e^{\gamma F(x,h,W_0)+f(G(x,W_0))}],\quad f\in \cC_{\omega}(\bR^{k}).
\end{equation}
\begin{remark}
Bellman equation \eqref{eq:rsc:bellmaneq2} is strictly connected to the Multiplicative Poisson Equation (MPE) defined for corresponding $\gamma$
(cf.
\cite{DiMSte2006} and references therein). Sufficient general conditions for which there exists a solution to MPE in the classic case (i.e.
using
ergodicity conditions and span norm or vanishing discount approach) could be found e.g. in \cite{DiMSte1999,KonMey2003,Her1989,HerMar1996}. For a more general conditions (obtained using splitting Markov techniques or Doeblin's condition) see e.g. \cite{DiMSte2006,CavHer2005}. Also using
robust
representation of the risk measure (i.e. $-\mu^{\gamma}$) \cite{FolSch2002}, one could notice that equation \eqref{eq:rsc:bellmaneq} corresponds
to
the Isaacs equation for ergodic cost stochastic dynamic game (cf. \cite{HerMar1996,FleHer1997} and references therein).
\end{remark}

\begin{proposition}\label{pr:RSC.feller}
Let $\gamma<0$. Under assumptions \eqref{as:rsc:A.1}--\eqref{as:rsc:A.4}, the operators $R_{\gamma}$ and $T_{\gamma}$ transforms the
set
$\cC_{\omega}(\bR^{k})$ into itself and for $f\in \cC_{\omega}(\bR^{k})$ the mapping $(-\infty,0)\times \bR^{k} \ni (\gamma,x) \mapsto T_\gamma f(x)$ is continuous.
\end{proposition}
\begin{proof}
We will only show the proof for $R_{\gamma}$, as the proof for $T_{\gamma}$ is analogous. Let $f\in\cC_{\omega}(\bR^{k})$ and $\gamma<0$. We know
that there exists $M>1$, such that for all $x\in\bR^{k}$, we get $|f(x)|\leq
M(\omega(x)+1)$.

First, let us prove that $\|R_{\gamma}f\|_{\omega}$ is finite. Using the fact that $\mu^{\gamma}$ is monotone and translation invariant as well
as
\eqref{as:rsc:A.4}, for any $x\in\bR^{k}$, we get
\begin{align*}
R_{\gamma}f(x) &\leq \mu^{\gamma}(a_2(W_0)+b_2\omega(x)+M(\omega(G(x,W_0))+1))\\
& \leq \mu^{\gamma}(a_2(W_0)+b_2\omega(x)+Ma_1(W_0)+Mb_1\omega(x)+M)\\
& = (b_2+Mb_1)\omega(x)+\mu^{\gamma}(a_2(W_0)+Ma_1(W_0))+M,
\end{align*}
as well as
\[
R_{\gamma}f(x) \geq -(b_2+Mb_1)\omega(x)+\mu^{\gamma}(-a_2(W_0)-Ma_1(W_0))-M.
\]
Consequently, noting that $R_{\gamma}f\in \cC_{\omega'}(\bR^{k})$ for
\[
\omega'(x)=(b_2+Mb_1)\omega(x)+|\mu^{\gamma}(a_2(W_0)+Ma_1(W_0))|+|\mu^{\gamma}(-a_2(W_0)-Ma_1(W_0))|+M,
\]
and using \eqref{eq:omega'}, we conclude that $\|R_{\gamma}f\|_{
\omega}$ is finite.

Second, let us prove that the mapping $(-\infty,0)\times \bR^{k} \ni (\gamma,x) \mapsto R_{\gamma}f(x)$ is continuous. Let $\{(\gamma_n,x_{n},h_n)\}_{n\in\bN}$ be a sequence such that $\gamma_n<0$ $x_n\in\bR^{k}$, $h_{n}\in U$ and
$(\gamma_n, x_{n},h_n)\to (\gamma, x,h)$, where $\gamma<0$, $x\in\bR^{k}$ and $h\in U$. By \eqref{as:rsc:A.2} and \eqref{as:rsc:A.3} we know that
\[e^{\gamma_n [F(x_n,h_n,W_0)+f(G(x_n,W_0))]}\stackrel{a.s.}{\longrightarrow} e^{\gamma [F(x,h,W_0)+f(G(x,W_0))]}.\]
As the weight function $\omega$ is continuous and finite-valued, we know that $y:=\sup_{n\in\bN}\omega(x_{n})<\infty$.
Moreover, using \eqref{as:rsc:A.4}, we get
\[
0\leq e^{\gamma_n [F(x_n,h_n,W_0)+f(G(x_n,W_0))]}\leq e^{\gamma_0 [a_2(W_0)+Ma_1(W_0)+(b_2+Mb_1)y+M]}
\]
with $\gamma_0$ such that for any $n$ we have $\gamma_n \leq \gamma_0$.
Noting that $e^{\gamma_0 [a_2(W_0)+Ma_1(W_0)+(b_2+Mb_1)y+M]}\in L^{1}$, by dominated convergence theorem,
\[\bE[e^{\gamma_n [F(x_n,h_n,W_0)+f(G(x_n,W_0))]}]\to \bE[e^{\gamma [F(x,h,W_0)+f(G(x,W_0))]}],\]
and consequently
\[\mu^{\gamma_n}(F(x_n,h_n,W_0)+f(G(x_n,W_0)))\to \mu^{\gamma}(F(x,h,W_0)+f(G(x,W_0))).\]
Let $h_{z}^\gamma := \argmax_{h\in U} \mu^{\gamma}(F(z,h,W_0)+f(G(z,W_0)))$, for any $z\in U$ (note that $U$ is compact). Due to continuity of the
function
$(\gamma, x,h)\mapsto  \mu^{\gamma}(F(x,h,W_0)+f(G(x,W_0)))$, we also know that
\[\mu^{\gamma_n}(F(x_n,h_{x_{n}}^{\gamma_n},W_0)+f(G(x_n,W_0)))\to \mu^{\gamma}(F(x,h_{x}^\gamma,W_0)+f(G(x,W_0))),\]
which imply continuity of $(\gamma,x) \to R_{\gamma}f(x)$.
\end{proof}

We are now ready to formulate the main result of this paper.
\begin{theorem}\label{th:RSC.contraction}
Let $\gamma<0$. Under assumptions \eqref{as:rsc:A.1}--\eqref{as:rsc:A.5}, for sufficiently small
$\beta>0$,
the operator $T_{\gamma}$ is a local contraction under $\|\cdot\|_{\beta,\omega\textrm{-span}}$, i.e. there exist functions $\beta: \bR_{+}\to (0,1)$ and $L: \bR_{+} \to (0,1)$ such that
\[
\| T_{\gamma}f_1-T_{\gamma}f_2\|_{\beta(M),\omega\textrm{-span}}\leq L(M)\|f_1-f_2\|_{\beta(M),\omega\textrm{-span}},
\]
for $f_1,f_2\in\cC_{\omega}(\bR^{k})$, such that $\| f_1\|_{\omega\textrm{-span}}\leq M$ and $\| f_2 \|_{\omega\textrm{-span}}\leq M$.
\end{theorem}
The proof of Theorem~\ref{th:RSC.contraction} will be split into three lemmas which we will now formulate and prove. Before we do this, let us introduce some helpful notation.

Let $(\Omega,\cF_{1},\bP_{1})$ be a probability space which corresponds to random variable $W_0$. For any
$f\in\cC_{\omega}(\bR^{k})$, $x\in \bR^{k}$ and $h\in U$ we will use the following notation
\begin{align}
h_{(x,f)} & :=\gamma\argmax_{h\in U}\mu^{\gamma}(F(x,h,W_0)+\frac{1}{\gamma}f(G(x,W_0)))\nonumber\\
&\phantom{:}=\argmin_{h\in U}\ln \bE[e^{\gamma F(x,h,W_0)+f(G(x,W_0))}],\label{eq:rsc:hy}\\
\bQ_{(x,f,h)} & := \gamma\argmin_{\bQ\in\cM_{1}}\Big[\bE_{\bQ}[F(x,h,W_0)+\frac{1}{\gamma}f(G(x,W_0))]-\frac{1}{\gamma}H[\bQ\|\bP_{1}]\Big]\nonumber\\
& \phantom{:}= \argmax_{\bQ\in\cM_{1}}\Big[\bE_{\bQ}[\gamma F(x,h,W_0)+f(G(x,W_0))]-H[\bQ\|
\bP_{1}]\Big]\label{eq:rsc:qy},
\end{align}
where $\cM_{1}:=\cM_{1}(\Omega,\cF_{1})$ denote the set of all probability measures on $(\Omega,\cF_{1})$ and $H[\bQ\|\bP_{1}]$ is the relative entropy of $\bQ$ wrt. $\bP_{1}$, i.e.
\[
H[\bQ\| \bP_{1}]:=
\begin{cases}
\bE_{\bQ}[\ln\frac{\d\bQ}{\d\bP_{1}}] & \textrm{if}\ \bQ\ll\bP_{1},\\
+\infty &\textrm{otherwise}.
\end{cases}
\]
Objects defined in \eqref{eq:rsc:hy} and \eqref{eq:rsc:qy} might be non-unique in the sense that $\argmin$ (or $\argmax$) might define a set, rather than a single element. Nevertheless, with slight abuse of notation, we take any fixed maximizer of \eqref{eq:rsc:hy} and assume that $h_{x,f}\in U$. To have a unique representation of measure $\bQ_{(x,f,h)}$, we use so called {\it Esscher transformation} \cite{Ger1979}. Before we write the explicit form of $\bQ_{(x,f,h)}$, let us give a more specific comment.
The measure $\bQ_{(x,f,h)}$ corresponds to the minimizing scenario in the robust (dual) representation of the entropic utility $\mu^{\gamma}$. Indeed (see e.g. \cite{PraMenRun1996}), for any $Z\in L^{0}(\Omega,\cF_{1},\bP_{1})$, such that $\gamma Ze^{\gamma Z}\in L^{1}(\Omega,\cF_{1},\bP_{1})$, we get
\begin{equation}\label{eq:robust.representation}
\mu^{\gamma}(Z)=\inf_{\bQ\in\cM_{1}}\Big[\bE_{\bQ}Z-\frac{1}{\gamma}H[\bQ\|\bP_{1}]\Big].
\end{equation}
To show that\[
Z=F(x,h,W_0)+\frac{1}{\gamma}f(G(x,W_0))
\]
is such that $\gamma Ze^{\gamma Z}\in L^{1}(\Omega,\cF_{1},\bP_{1})$, it is enough to note that $\|f\|_{\omega}<\infty$ and use \eqref{as:rsc:A.4}. Then, we get
\[
Z\in L^{1}(\Omega,\cF_1,\bP_1)\quad\textrm{and}\quad e^{2\gamma Z}\in L^{1}(\Omega,\cF_1,\bP_1),
\]
which combined with the fact that for any $\gamma<0$ we get
\[
|\gamma Ze^{\gamma Z}|\leq \1_{\{\gamma Z \leq 0\}} |\gamma Z|+\1_{\{\gamma Z> 0\}}|e^{2\gamma Z}|,
\]
concludes the proof.
Then, as shown in \cite[Proposition 2.3]{PraMenRun1996}, we could define the minimizer of \eqref{eq:rsc:qy} through  {\it Esscher transformation} of $Z$, i.e. the measure $\bQ_{(x,f,h)}$ given by
\begin{equation}\label{eq:rsc:essher}
\bQ_{(x,f,h)}(dw)=\frac{e^{\gamma F(x,h,w)+f(G(x,w))}\bP_{1}(dw)}{\bE[e^{\gamma F(x,h,W_0)+f(G(x,W_0))}]}.
\end{equation}
We will also define the measure $\bar{\bQ}_{(x,f,h)}$ on $\bR^{k}$, by
\begin{equation}\label{eq:rsc:essher2}
\bar{\bQ}_{(x,f,h)}(A)=\frac{\bE\big[\1_{\set{G(x,W_0)\in
A}}e^{\gamma F(x,h,W_0)+f(G(x,W_0))}\big]}{\bE[e^{\gamma F(x,h,W_0)+f(G(x,W_0))}]},\quad
A\in\cB(\bR^{k}).
\end{equation}
Finally, for any $f,g\in\cC_{\omega}(\bR^{k})$ and $x,y\in\bR^{k}$ we shall write
\begin{equation}\label{eq:Hfgxy}
\bH^{f,g}_{x,y}:=\bar{\bQ}_{(x,f,h_{(x,g)})}-\bar{\bQ}_{(y,g,h_{(y,f)})}.
\end{equation}

We are now ready to introduce Lemma \ref{lem:1}, Lemma \ref{lem:2} and Lemma \ref{lem:3}.
\begin{lemma}\label{lem:1}
Let $\gamma<0$. Under assumptions \eqref{as:rsc:A.1}--\eqref{as:rsc:A.4}, we get
\begin{equation}\label{eq:w.3}
T_{\gamma}f(x)-T_{\gamma}g(x)-(T_{\gamma}f(y)-T_{\gamma}g(y))  \leq
\|f-g\|_{\beta,\omega\textrm{-span}}\|\bH^{f,g}_{x,y} \|_{\beta,\omega\textrm{-var}},
\end{equation}
 for any $f,g\in\cC_{\omega}(\bR^{k})$, $x,y\in\bR^{k}$ and $\beta>0$.
\end{lemma}
\begin{proof}
Let $f,g\in\cC_{\omega}(\bR^{k})$, $x,y\in\bR^{k}$ and let $\beta>0$. Using \eqref{eq:rsc:hy} we get
\begin{align}
T_{\gamma}f(x) &= \gamma\sup_{h\in U}\mu^{\gamma}(F(x,h,W_0)+\frac{1}{\gamma}f(G(x,W_0)))\nonumber\\
&\leq \gamma \mu^{\gamma}(F(x,h_{(x,g)},W_0)+\frac{1}{\gamma}f(G(x,W_0)))\nonumber\\
&=  \sup_{\bQ\in\cM_{1}(\bP_{1})} \Big[\bE_{\bQ}[\gamma F(x,h_{(x,g)},W_0)+f(G(x,W_0))]-H[\bQ\|
\bP_{1}]\Big]\nonumber\\
&= \bE_{\bQ_{(x,f,h_{(x,g)})}}\left[ \gamma F(x,h_{(x,g)},W_0)+f(G(x,W_0))\right] -H[\bQ_{(x,f,h_{(x,g)})}\| \bP_{1}]\label{eq:rsc:tf}
\end{align}

Now, using \eqref{eq:rsc:qy} we get
\begin{align}
T_{\gamma}g(x)& =\gamma \sup_{h\in U}\mu^{\gamma}(F(x,h,W_0)+\frac{1}{\gamma}g(G(x,W_0)))\nonumber\\
& = \gamma\mu^{\gamma}(F(x,h_{(x,g)},W_0)+\frac{1}{\gamma}g(G(x,W_0)))\nonumber\\
 &= \sup_{\bQ\in\cM_{1}(\bP_{1})}\Big[\bE_{\bQ}[\gamma F(x,h_{(x,g)},W_0)+g(G(x,W_0))]-H[\bQ\|
 \bP_{1}]\Big]\nonumber\\
 &\geq \bE_{\bQ_{(x,f,h_{(x,g)})}}\left[\gamma F(x,h_{(x,g)},W_0)+g(G(x,W_0))\right]-H[\bQ_{(x,f,h_{(x,g)})}\| \bP_{1}]
\label{eq:rsc:tg}
\end{align}
Combining \eqref{eq:rsc:tf} and \eqref{eq:rsc:tg} we get
\begin{align}
T_{\gamma}f(x)-T_{\gamma}g(x) & \leq \bE_{\bQ_{(x,f,h_{(x,g)})}}[f(G(x,W_0))-g(G(x,W_0))]\nonumber\\
& \leq \int_{\bR^{k}} [f(z)-g(z)]\bar{\bQ}_{(x,f,h_{(x,g)})}(dz).\label{eq:rsc:tfx1tfx2}
\end{align}
Switching $f$ with $g$ in \eqref{eq:rsc:tfx1tfx2}, and doing similar computations for $y\in\bR^{k}$, we get
\begin{equation}\label{eq:rsc:tfx2tfx1}
T_{\gamma}g(y)-T_{\gamma}f(y) \leq  \int_{\bR^{k}} [g(z)-f(z)]\bar{\bQ}_{(y,g,h_{(y,f)})}(dz)
\end{equation}
Combining \eqref{eq:rsc:tfx1tfx2} with \eqref{eq:rsc:tfx2tfx1} and recalling notation \eqref{eq:Hfgxy}, we get
\begin{equation}\label{eq:main1}
T_{\gamma}f(x)-T_{\gamma}g(x)-(T_{\gamma}f(y)-T_{\gamma}g(y)) \leq \int_{\bR^{k}}
\big[f(z)-g(z)\big]\bH^{f,g}_{x,y}(dz).
\end{equation}
We know that for any $c\in\bR$, we get
\[
\int_{\bR^{k}} \big[f(z)-g(z)\big]\bH^{f,g}_{x,y}(dz)=\int_{\bR^{k}} \frac{f(z)-g(z)+c}{1+\beta\omega(z)}(1+\beta\omega(z))\bH^{f,g}_{x,y}(dz).
\]
Let $A\subset \bR^{k}$ denote a positive set for a signed measure $\bH^{f,g}_{x,y}$ (obtained e.g. using Hahn-Jordan decomposition) and for any $c\in\bR$ let
\[
a_{+}(c):=\sup_{z\in\bR^{k}}\frac{f(z)-g(z)+c}{1+\beta\omega(z)}\quad\textrm{and}\quad
a_{-}(c):=-\inf_{z\in\bR^{k}}\frac{f(z)-g(z)+c}{1+\beta\omega(z)}.
\]
Then, for any $c\in\bR$, we get
\begin{equation}\label{eq:w.1}
\int_{\bR^{k}} \big[f(z)-g(z)\big]\bH^{f,g}_{x,y}(dz) \leq a_{+}(c)\int_{A}(1+\beta\omega(z))\bH^{f,g}_{x,y}(dz)-a_{-}(c)\int_{A^{c}}(1+\beta\omega(z))\bH^{f,g}_{x,y}(dz).
\end{equation}
From Proposition~\ref{pr:omega} we know that there exists $c_0\in\bR$, such that \[
a_{+}(c_0)=a_{-}(c_0)=\|f-g\|_{\beta,\omega\textrm{-span}}.
\]
Thus, from \eqref{eq:w.1} we get
\begin{equation}\label{eq:w.4}
\int_{\bR^{k}} \big[f(z)-g(z)\big]\bH^{f,g}_{x,y}(dz) \leq \|f-g\|_{\beta,\omega\textrm{-span}}\|\bH^{f,g}_{x,y}\|_{\beta,\omega\textrm{-var}},
\end{equation}
which together with \eqref{eq:main1} concludes the proof of \eqref{eq:w.3}.

\end{proof}

\begin{lemma}\label{lem:2}
Let $\gamma<0$. Under assumptions \eqref{as:rsc:A.1}--\eqref{as:rsc:A.4}, for any fixed $M>0$ and $\phi\in (b_1,1)$, there exists $\alpha_{\phi}>0$, such that
\begin{equation}\label{eq:beta.var.leq.var}
\|\bH^{f,g}_{x,y} \|_{\beta, \omega\textrm{-var}} \leq
\|\bH^{f,g}_{x,y} \|_{\textrm{var}}+\beta(\phi\omega(x)+\phi\omega(y)+2\alpha_{\phi}),
\end{equation}
for any $x,y\in\bR^{k}$ and $f,g\in\cC_{\omega}(\bR^{k})$ satisfying $\| f\|_{\omega\textrm{-span}}\leq M$ and  $\| g\|_{\omega\textrm{-span}}\leq M$.
\end{lemma}

\begin{proof}
For any $x,y\in\bR^{k}$ and $f,g\in\cC_{\omega}(\bR^{k})$ we get
\begin{align*}
\|\bH^{f,g}_{x,y} \|_{\beta, \omega\textrm{-var}} & =\int_{\bR^{k}}\big(1+\beta\omega(z)\big)|\bH^{f,g}_{x,y}|(dz)\\
& = \int_{\bR^{k}}|\bH^{f,g}_{x,y}|(dz)+\beta\int_{\bR^{k}}\omega(z)|\bH^{f,g}_{x,y}|(dz)\\
& \leq \|\bH^{f,g}_{x,y} \|_{\textrm{var}}+\beta\left(\int_{\bR^{k}}\omega(z)\bar{\bQ}_{(x,f,h_{(x,g)})}(dz)+\int_{\bR^{k}}\omega(z)\bar{\bQ}_{(y,g,h_{(y,f)})}(dz)\right).
\end{align*}
Thus, to prove \eqref{eq:beta.var.leq.var} it is sufficient to show that for any fixed $M>0$ and $\phi\in (b_1,1)$, there exists $\alpha_{\phi}>0$, such that
\begin{equation}\label{eq:omegaK}
\int_{\bR^{k}}\omega(z) \bar{\bQ}_{(x,f,h)}(dz)\leq \phi\omega(x)+\alpha_{\phi},
\end{equation}
for any $h\in U$, $x\in\bR^{k}$ and $f\in\cC_{\omega}(\bR^{k})$ satisfying $\| f\|_{\omega\textrm{-span}}\leq M$.

Let $M>0$ and $\phi\in (b_1,1)$. Using \eqref{eq:rsc:essher} and \eqref{eq:rsc:essher2} we get that \eqref{eq:omegaK} is equivalent to
\[
\bE\left[\left(\omega(G(x,W_0))-\phi\omega(x)\right)e^{\gamma F(x,h,W_0)+f(G(x,W_0))}\right] \leq
\alpha_{\phi}\bE\left[e^{\gamma F(x,h,W_0)+f(G(x,W_0))}\right].
\]
For simplicity let $Z:= \gamma F(x,h,W_0)+f(G(x,W_0))$. It is enough to prove that
\[
\bE\left[\1_{A}\left(\omega(G(x,W_0))-\phi\omega(x)\right)e^{Z}\right] \leq \frac{\alpha_{\phi}}{2}\bE\left[e^{Z}\right],
\]
where $A=\{\omega(G(x,W_0))-\phi\omega(x)>\frac{\alpha_{\phi}}{2}\}$, as the inequality
\[
\bE\left[\1_{A^{c}}\left(\omega(G(x,W_0))-\phi\omega(x)\right)e^{Z}\right] \leq \frac{\alpha_{\phi}}{2}\bE\left[e^{Z}\right]
\]
is trivial. Using Schwarz inequality we get $1\leq \bE[e^{-Z}]\bE[e^{Z}]$, so it is enough to show that
\begin{equation}\label{eq:7.K1}
\bE\left[\1_{A}\left(\omega(G(x,W_0))-\phi\omega(x)\right)e^{Z}\right] \bE\left[e^{-Z}\right]\leq \frac{\alpha_{\phi}}{2},
\end{equation}
Multiplying both sides of \eqref{eq:7.K1} by $\frac{2(Mb_1-\gamma b_2)}{(\phi-b_1)}$, using the fact that $y<e^{y}$ for any $y>0$, and inequality
$\frac{2Mb_1}{(\phi-b_1)}<\frac{2(Mb_1-\gamma b_2)}{(\phi-b_1)}$, to prove
\eqref{eq:omegaK}, it is sufficient to show that
\begin{equation}\label{eq:1111111}
\bE\left[e^{\frac{2(Mb_1-\gamma b_2)}{(\phi-b_1)}(\omega(G(x,W_0))-\phi\omega(x))}e^{Z}\right] \bE\left[e^{- Z}\right] \leq
\frac{\alpha_{\phi}Mb_1}{(\phi-b_1)}.
\end{equation}
Using \eqref{as:rsc:A.4} and Schwarz inequality we get
\begin{align*}
\bE\left[e^{\frac{2(Mb_1-\gamma b_2)}{(\phi-b_1)}(\omega(G(x,W_0))-\phi\omega(x))}e^{Z}\right] &
\leq \bE\left[e^{\frac{2(Mb_1-\gamma b_2)}{(\phi-b_1)} \left[a_1(W_0)-(\phi-b_1)\omega(x)\right]}e^{Z}\right]\\
& \leq e^{-2(Mb_1-\gamma b_2)\omega(x)}\bE\left[e^{\frac{2(Mb_1-\gamma b_2)}{(\phi-b_1)} a_1(W_0)}e^{Z}\right]\\
& \leq e^{-2(Mb_1-\gamma b_2)\omega(x)}\sqrt{\bE[e^{\frac{4(Mb_1-\gamma b_2)}{(\phi-b_1)} a_1(W_0)}]}\sqrt{\bE[e^{2Z}]},
\end{align*}
so instead of \eqref{eq:1111111} it is enough to show that
\begin{equation}\label{eq:111}
e^{-2(Mb_1-\gamma b_2)\omega(x)}\sqrt{\bE[e^{\frac{4(Mb_1-\gamma b_2)}{(\phi-b_1)} a_1(W_0)}]}\sqrt{\bE[e^{2Z}]}\bE\left[e^{-Z}\right] \leq \frac{\alpha_{\phi}Mb_1}{(\phi-b_1)}.
\end{equation}
Let us prove \eqref{eq:111}. Due to \eqref{as:rsc:A.4} we know that
\begin{equation}\label{eq:C1}
\sqrt{\bE[e^{\frac{4(Mb_1-\gamma b_2)}{(\phi-b_1)} a_1(W_0)}]}<\infty.
\end{equation}
On the other hand, from the fact that $\|f\|_{\omega\textrm{-span}}\leq M$, we know that there exists $a\in\bR$ such that $\| f+a \|_{\omega}\leq
M$.
Consequently, recalling that $Z=\gamma F(x,h,W_0)+f(G(x,W_0))$, using monotonicity of the exponent function  and \eqref{as:rsc:A.4}, we get
\begin{align}
\sqrt{\bE[e^{2Z}]} & = \sqrt{\bE[e^{2[\gamma F(x,h,W_0)+(f(G(x,W_0))+a)-a]}]}\nonumber \\
& \leq \sqrt{\bE[e^{2[-\gamma a_2(W_0)-\gamma b_{2}\omega(x)+M(a_1(W_0)+b_{1}\omega(x)+1)-a]}]}\nonumber\\
& = e^{(Mb_{1}-\gamma b_2)\omega(x)+M-a}\sqrt{\bE[e^{2[Ma_1(W_0)-\gamma a_2(W_0)]}]},\label{eq:ineq1.aa1}\\
\bE[e^{-Z}] & = \bE[e^{-[\gamma F(x,h,W_0)+(f(G(x,W_0))+a)-a]}]\nonumber\\
& \leq \bE[e^{-\gamma a_2(W_0)-\gamma b_{2}\omega(x)+M(a_1(W_0)+b_{1}\omega(x)+1)+a]}]\nonumber\\
& = e^{(Mb_1-\gamma b_2)\omega(x)+M+a}\bE[e^{Ma_{1}(W_0)-\gamma a_2(W_0)}].\label{eq:ineq1.aa2}
\end{align}
Using \eqref{eq:ineq1.aa1}, \eqref{eq:ineq1.aa2} and \eqref{eq:assumpt.kpm} we get
\begin{equation}\label{eq:K2.1}
e^{-2(Mb_1-\gamma b_2)\omega(x)}\sqrt{\bE[e^{2Z}]}\bE\left[e^{-Z}\right]=e^{2M}\sqrt{\bE[e^{2[Ma_1(W_0)-\gamma
a_2(W_0)]}]}\bE[e^{Ma_1(W_0)-\gamma a_2(W_0)}]<\infty.
\end{equation}
Combining \eqref{eq:K2.1} and \eqref{eq:C1}, we get that \eqref{eq:111} will hold for $\alpha_{\phi}$ large enough. In other words it is enough to choose $\alpha_{\phi}$, such that
\begin{equation}\label{eq:alpha.phi}
\frac{e^{2M}(\phi-b_1)}{Mb_1}\sqrt{\bE[e^{\frac{4(Mb_1-\gamma b_2)}{(\phi-b_1)} a_1(W_0)}]}\sqrt{\bE[e^{2[Ma_1(W_0)-\gamma
a_2(W_0)]}]}\bE[e^{Ma_1(W_0)-\gamma a_2(W_0)}] \leq
\alpha_{\phi}.
\end{equation}
This concludes the proof of \eqref{eq:omegaK}.
\end{proof}

\begin{lemma}\label{lem:3}
Let $\gamma<0$. Under assumptions \eqref{as:rsc:A.1}--\eqref{as:rsc:A.5}, for any fixed $M>0$, $\phi\in (b_1,1)$ and $\alpha_{\phi}>0$, there exists $\beta\in (0,1)$ and $L\in (0,1)$ such that
\begin{equation}\label{eq:rsc:var.leq}
\|\bH^{f,g}_{x,y} \|_{\textrm{var}}+\beta(\phi\omega(x)+\phi\omega(y)+2\alpha_{\phi})\leq L(2+\beta\omega(x)+\beta\omega(y)),
\end{equation}
for any $x,y\in\bR^{k}$ and $f,g\in\cC_{\omega}(\bR^{k})$ satisfying $\| f\|_{\omega\textrm{-span}}\leq M$ and $\| g \|_{\omega\textrm{-span}}\leq M$.
\end{lemma}
\begin{proof}
Let us fix $M>0$, $\phi\in (b_1,1)$ and $\alpha_{\phi}>0$. Let $R\in\bR$ be such that
\begin{equation}\label{eq:Rinitial}
R>\frac{2\alpha_{\phi}}{1-\phi}.
\end{equation}
We will consider two cases:
\[
\textrm{(a)}\quad \omega(x)+\omega(y)>R,\quad\quad \textrm{(b)}\quad \omega(x)+\omega(y)\leq R,
\]
and find $\beta<1$ and $L\in (0,1)$ such that \eqref{eq:rsc:var.leq} is satisfied both on $\{\omega(x)+\omega(y)>R\}$ and $\{\omega(x)+\omega(y)\leq R\}$.


\begin{enumerate}[{\it Case} a)]

\item Noting that $\|\bH^{f,g}_{x,y}\|_{\textrm{var}}\leq 2$, it is enough to find $\beta<1$ and $L\in (0,1)$ such that
\begin{equation}\label{eq:case1.1}
2+\beta(\phi\omega(x)+\phi\omega(y)+2\alpha_{\phi})\leq L(2+\beta\omega(x)+\beta\omega(y)),
\end{equation}
for any $x,y\in\bR^{k}$, such that $\omega(x)+\omega(y)>R$.
We will show that in this case for any $\beta<1$ we could find $L\in (0,1)$ such that \eqref{eq:case1.1} holds. Let $\beta<1$. We know that
\eqref{eq:case1.1} is equivalent to
\[
2+2\beta \alpha_{\phi} \leq 2L+\beta(L-\phi)(\omega(x)+\omega(y)).
\]
Let us assume that $L>\phi$. Then, it is sufficient to show that
\[
2+2\beta\alpha_{\phi} \leq 2L+\beta(L-\phi)R,
\]
which is equivalent to
\begin{equation}\label{eq:suff.L}
\frac{2+\beta(2\alpha_{\phi}+\phi R)}{2+\beta R}\leq L.
\end{equation}
Consequently, using \eqref{eq:Rinitial}, it is enough to choose any $L<1$ such that
\begin{equation}\label{eq:L1}
L\in \left(\max\left\{\phi,\frac{2+\beta(2\alpha_{\phi}+\phi R)}{2+\beta R}\right\},1\right).
\end{equation}

\item Let  $C_{R}:=\{(x,y)\in \bR^{k}\times\bR^{k}: \omega(x)+\omega(y)\leq R\}$.
It is sufficient to show that there exists $\beta\in (0,1)$ and $L\in (0,1)$ such that for any $(x,y)\in C_{R}$ and $f,g\in\cC_{\omega}(\bR^{k})$ satisfying $\| f\|_{\omega\textrm{-span}}\leq M$ and $\| g \|_{\omega\textrm{-span}}\leq M$, we get
\[
\|\bH^{f,g}_{x,y} \|_{\textrm{var}}+\beta(\phi R+2\alpha_{\phi}) <2L.
\]
In fact, it is enough to show that
\begin{equation}\label{eq:rsc:var.leq2}
\sup_{(x,y)\in C_{R}}\|\bH^{f,g}_{x,y} \|_{\textrm{var}}<2.
\end{equation}
Indeed, then it is enough to choose any $\beta<1$ such that
\[
\beta < \frac{2-\sup_{(x,y)\in C_{R}}\|\bH^{f,g}_{x,y} \|_{\textrm{var}}}{\phi R+2\alpha_{\phi}},
\]
and consider any
\begin{equation}\label{eq:L2}
L\in\left( \frac{\sup_{(x,y)\in C_{R}}\|\bH^{f,g}_{x,y} \|_{\textrm{var}}+\beta (\phi
R+2\alpha_{\phi})}{2},1\right).
\end{equation}
On the contrary, let us assume that \eqref{eq:rsc:var.leq2} is false. Then, there exists a sequence
\[(x_n,y_n,f_n,g_n,A_n)_{n\in\bN},\]
for $(x_n,y_n)\in C_{R}$, $f_n,g_n\in\cC_{\omega}(\bR^{k})$ and $A_n\in\cB(\bR^{k})$, such that $\| f_n\|_{\omega\textrm{-span}}\leq M$, $\|
g_n
\|_{\omega\textrm{-span}}\leq M$ and
\begin{equation}\label{eq:rsc2:seqn}
\bH^{f_n,g_n}_{x_n,y_n}(A_n)=\bar{\bQ}_{(x_n,g_n,h_{(x_n,f_n)})}(A_n)-\bar{\bQ}_{(y_n,f_n,h_{(y_n,g_n)})}(A_n)\rightarrow 1.
\end{equation}
Due to \eqref{eq:rsc2:seqn} we know that
\begin{equation}\label{eq:rsc2:seqn2}
\bar{\bQ}_{(x_n,g_n,h_{(x_n,f_n)})}(A^c_n)\rightarrow 0\quad\textrm{and}\quad \bar{\bQ}_{(y_n,f_n,h_{(y_n,g_n)})}(A_n)\rightarrow 0.
\end{equation}
Next, for any $x\in\bR^{k}$, $h\in U$, $f\in\cC_{\omega}(\bR^{k})$ and $A\in\cB(\bR^{k})$, such that $\omega(x)\leq R$ and $\|
f\|_{\omega\textrm{-span}}\leq M$, using Schwarz inequality we get
\begin{align}
\bar{\bQ}_{(x,f,h)}(A) &= \frac{\bE\big[\1_{\set{G(x,W_0)\in
A}}e^{\gamma[F(x,h,W_0)+\frac{1}{|\gamma|}f(G(x,W_0))]}\big]}{\bE[e^{\gamma[F(x,h,W_0)+\frac{1}{|\gamma|}f(G(x,W_0))]}]}\nonumber\\
& =\frac{\bE\big[\1_{\set{G(x,W_0)\in
A}}e^{\gamma[F(x,h,W_0)+\frac{1}{|\gamma|}f(G(x,W_0))]}\big]}{\bE[e^{\gamma[F(x,h,W_0)+\frac{1}{|\gamma|}f(G(x,W_0))]}]}\frac{\bE[e^{-\gamma[F(x,h,W_0)+\frac{1}{|\gamma|}f(G(x,W_0))]}]}{\bE[e^{-\gamma[F(x,h,W_0)+\frac{1}{|\gamma|}f(G(x,W_0))]}]}\nonumber\\
& \geq \frac{\bE\big[\1_{\set{G(x,W_0)\in
A}}e^{\frac{\gamma}{2}[F(x,h,W_0)+\frac{1}{|\gamma|}f(G(x,W_0))]}e^{-\frac{\gamma}{2}[F(x,h,W_0)+\frac{1}{|\gamma|}f(G(x,W_0))]}\big]^{2}}{\bE[e^{\gamma[F(x,h,W_0)+\frac{1}{|\gamma|}f(G(x,W_0))]}]\bE[e^{-\gamma[F(x,h,W_0)+\frac{1}{|\gamma|}f(G(x,W_0))]}]}\nonumber\\
& \geq \frac{\bE\big[\1_{\set{G(x,W_0)\in A}}\big]^{2}}{e^{2[(Mb_1-\gamma b_2)\omega(x)+M]}\bE[e^{Ma_1(W_0)-\gamma a_2(W_0)}]^{2}}\nonumber\\
& \geq \frac{\bE\big[\1_{\set{G(x,W_0)\in A}}\big]^{2}}{e^{2[(Mb_1-\gamma b_2)R+M]}\bE[e^{Ma_1(W_0)-\gamma
a_2(W_0)}]^{2}}\label{eq:rsc2:schwarz}.
\end{align}
Combining \eqref{eq:rsc2:seqn2} and \eqref{eq:rsc2:schwarz}, we get that
\[
\bE\big[\1_{\set{G(x_n,W_0)\in A^{c}_n}}\big]\to 0\quad\textrm{and}\quad \bE\big[\1_{\set{G(y_n,W_0)\in A_n}}\big]\to 0.
\]
On the other hand, from \eqref{as:rsc:A.5}, for any $n\in\bN$ and $(x_{n},y_{n})\in C_{R}$, we get
\[
\bE\big[\1_{\set{G(x_n,W_0)\in A^{c}_n}}\big]+\bE\big[\1_{\set{G(y_n,W_0)\in A_n}}\big]\geq c\nu(A_{n}^{c})+c\nu(A_{n})=c>0,
\]
where $c$ and $\nu$ satisfy \eqref{eq:rsc:unerg}, for $C_{R}$. This leads to contradiction and in consequence concludes the proof of Case b).
\end{enumerate}

We are now ready to prove \eqref{eq:rsc:var.leq}. Indeed, combining \eqref{eq:L1} and \eqref{eq:L2} we conclude that for a given $M>0$, $\phi\in (b_1,1)$, $\alpha_{\phi}>0$ and $R\in\bR$ satisfying \eqref{eq:Rinitial}, it is enough to choose $\beta<1$ and $L\in (0,1)$, such that
\begin{align}
\beta & < \frac{2-\sup_{(x,y)\in C_{R}}\|\bH^{f,g}_{x,y} \|_{\textrm{var}}}{\phi
R+2\alpha_{\phi}},\nonumber\\
L & > \max \left\{\phi\, ,\, \frac{\sup_{(x,y)\in C_{R}}\|\bH^{f,g}_{x,y} \|_{\textrm{var}}+\beta (\phi
R+2\alpha_{\phi})}{2},\frac{2+\beta(2\alpha_{\phi}+\phi R)}{2+\beta R}\right\}\label{eq:BetaAndL}.
\end{align}
This concludes the proof of \eqref{eq:rsc:var.leq}.
\end{proof}

We are now ready to prove Theorem~\ref{th:RSC.contraction}.

\begin{proof}[Proof of Theorem~\ref{th:RSC.contraction}]
Let $\gamma<0$. Combining Lemma \ref{lem:1}, Lemma \ref{lem:2} and Lemma \ref{lem:3} we know that for any fixed $M$, there exists $\beta(M)\in (0,1)$ and $L(M)\in (0,1)$, such that
\begin{align*}
\frac{T_{\gamma}f(x)-T_{\gamma}g(x)-(T_{\gamma}f(y)-T_{\gamma}g(y))}{2+\beta(M)\omega(x)+\beta(M)\omega(y)}  & \leq
\frac{\|f-g\|_{\beta(M),\omega\textrm{-span}}\|\bH^{f,g}_{x,y} \|_{\beta(M),\omega\textrm{-var}}}{2+\beta(M)\omega(x)+\beta(M)\omega(y)}\\
& \leq L(M)\|f-g\|_{\beta(M),\omega\textrm{-span}},
\end{align*}
for any $f,g\in\cC_{\omega}(\bR^{k})$ and $x,y\in\bR^{k}$ satisfying $\| f\|_{\omega\textrm{-span}}\leq M$ and $\| g \|_{\omega\textrm{-span}}\leq M$.
Consequently, for any fixed $M$, there exists $\beta(M)\in (0,1)$ and $L(M)\in (0,1)$, such that
\[
\| T_{\gamma}f -T_{\gamma}g\|_{\beta(M),\omega\textrm{-span}}\leq L(M)\|f-g\|_{\beta(M),\omega\textrm{-span}},
\]
whenever $\| f\|_{\omega\textrm{-span}}\leq M$ and $\| g \|_{\omega\textrm{-span}}\leq M$. This concludes the proof of Theorem \ref{th:RSC.contraction}.
\end{proof}

\begin{corollary}\label{cor:RSC.contraction2}
 For a given $\gamma_0<0$ there exists $\beta: \bR_{+}\to (0,1)$ and $L: \bR_{+} \to (0,1)$, such that  for any $\gamma\in [\gamma_0,0)$, operator
$T_{\gamma}$ is a local contraction wrt. $\beta$ and $L$, i.e. for any $\gamma\in [\gamma_0,0)$, we get
\[
\| T_{\gamma}f_1-T_{\gamma}f_2\|_{\beta(M),\omega\textrm{-span}}\leq L(M)\|f_1-f_2\|_{\beta(M),\omega\textrm{-span}},
\]
for $f_1,f_2\in\cC_{\omega}(\bR^{k})$, such that $\| f_1\|_{\omega\textrm{-span}}\leq M$ and $\| f_2 \|_{\omega\textrm{-span}}\leq M$.
\end{corollary}

\begin{proof}
The proof of Corollary \ref{cor:RSC.contraction2} is a direct consequence of the proof of Theorem~\ref{th:RSC.contraction}. For
transparency, let us briefly explain the idea of the proof.

For clarity let us fix $M>0$ and consider $L(M)\in (0,1)$ and $\beta(M)\in (0,1)$. Let
$\alpha_{\phi}>0$ be such that \eqref{eq:alpha.phi} is satisfied for $\gamma_0$, i.e.
\[
\alpha_{\phi}\geq \frac{e^{2M}(\phi-b_1)}{Mb_1}\sqrt{\bE[e^{\frac{4(Mb_1-\gamma_0 b_2)}{(\phi-b_1)} a_1(W_0)}]}\sqrt{\bE[e^{2[Ma_1(W_0)-\gamma_0
a_2(W_0)]}]}\bE[e^{Ma_1(W_0)-\gamma_0 a_2(W_0)}],
\]
and let $R$ be such \eqref{eq:Rinitial} is satisfied for $\gamma_0$. Then, for any $\gamma\in[\gamma_0,0)$ we get
\[
\alpha_{\phi}\geq \frac{e^{2M}(\phi-b_1)}{Mb_1}\sqrt{\bE[e^{\frac{4(Mb_1-\gamma b_2)}{(\phi-b_1)} a_1(W_0)}]}\sqrt{\bE[e^{2[Ma_1(W_0)-\gamma
a_2(W_0)]}]}\bE[e^{Ma_1(W_0)-\gamma a_2(W_0)}].
\]
Consequently, the choice of $\alpha_{\phi}$ and $R$ will guarantee \eqref{eq:alpha.phi} and \eqref{eq:Rinitial}, for any $\gamma\in [\gamma_0,0)$.

Next, we know that $\beta(M)$ and $L(M)$ are chosen in such a way that \eqref{eq:BetaAndL} is satisfied for $\gamma_0$, i.e.
\begin{align*}
\beta & < \frac{2-\sup_{(x,y)\in C_{R}}\|\bH^{f,g}_{x,y} \|_{\textrm{var}}}{\phi R+2\alpha_{\phi}},\\
L & > \max \left\{\phi\, ,\, \frac{\sup_{(x,y)\in C_{R}}\|\bH^{f,g}_{x,y} \|_{\textrm{var}}+\beta (\phi
R+2\alpha_{\phi})}{2},\frac{2+\beta(2\alpha_{\phi}+\phi R)}{2+\beta R}\right\}.
\end{align*}
Thus, it is sufficient to show that we could find a constant $a\in (0,2)$ such that
\[
\sup_{(x,y)\in C_{R}}\|\bH^{f,g}_{x,y} \|_{\textrm{var}}\leq a
\]
for any $\gamma\in [\gamma_0,0)$. To do that it is enough to notice that
the lower bound for $\bar{\bQ}_{(x,f,h)}$ introduced in \eqref{eq:rsc2:schwarz} is in fact decreasing wrt. $\gamma$.
\end{proof}
Using Theorem~\ref{th:RSC.contraction}, i.e. contraction property of operator $T_{\gamma}$, one can solve Bellman equation~\eqref{eq:rsc:bellmaneq2} and \eqref{eq:rsc:bellmaneq}.
\begin{proposition}\label{pr:RSC.Bellman.solution}
Under assumptions \eqref{as:rsc:A.1}--\eqref{as:rsc:A.5}, there exists $\gamma_0<0$, such that for any $\gamma\in (\gamma_0,0)$, there exist a unique (up to an additive constant) $u_{\gamma}\in\cC_{\omega}(\bR^k)$ and $\lambda_{\gamma}\in\bR$, the solutions to Bellman
equation~\eqref{eq:rsc:bellmaneq2}.
\end{proposition}

\begin{proof} Let us fix $\bar{\gamma}<0$ and let $M:=\mu^{0}(a_{2}(W_0))-\mu^{\bar{\gamma}}(-a_{2}(W_0))+b_2$. We know that for any $\gamma\in
[\bar{\gamma},0)$ we get $\| R_{\gamma}0\|_{\omega\textrm{-span}}\leq M$, as
\begin{align*}
\| R_{\gamma}0\|_{\omega\textrm{-span}}
 & \leq \sup_{x,y\in \bR^{k}}\frac{\mu^{\gamma}(a_{2}(W_0)+b_2\omega(x))-\mu^{\gamma}(-a_{2}(W_0)-b_2\omega(y))}{2+\omega(x)+\omega(y)}\\
 & \leq \sup_{x,y\in \bR^{k}}\frac{\mu^{0}(a_{2}(W_0))-\mu^{\bar{\gamma}}(-a_{2}(W_0))+b_2\omega(x)+b_2\omega(y)}{2+\omega(x)+\omega(y)}\\
& \leq  \mu^{0}(a_{2}(W_0))-\mu^{\bar{\gamma}}(-a_{2}(W_0))+b_2.
\end{align*}
For the operator $T_{\bar{\gamma}}$ and $M$, let $\beta(M)$ and $L(M)$ denote corresponding constants from Theorem~\ref{th:RSC.contraction}. For simplicity we will write $\beta$ and $L$, instead of $\beta(M)$ and $L(M)$. Let
\begin{equation}\label{eq:gamma0}
\gamma_0:=\max\{\bar{\gamma},-|\beta(1-L)|\}
\end{equation}
Noting that $\gamma_0\in (-1,0)$ and using Corollary~\ref{cor:RSC.contraction2}, for any $\gamma\in (\gamma_0,0)$, we know that
\begin{equation}\label{eq:prop5.4.a}
\| T_{\gamma}f_1-T_{\gamma}f_2\|_{\beta,\omega\textrm{-span}}\leq L\|f_1-f_2\|_{\beta,\omega\textrm{-span}},
\end{equation}
for $f_1,f_2\in\cC_{\omega}(\bR^{k})$, such that $\| f_1\|_{\omega\textrm{-span}}\leq M$ and $\| f_2 \|_{\omega\textrm{-span}}\leq M$.

As $|\gamma|<\beta(1-L)$, it can be easily shown that for any $n\in\bN$ we get $\| T^{n}_{\gamma}0\|_{\omega\textrm{-span}}\leq M$.
Indeed, using \eqref{eq:prop5.4.a}, we get
\begin{align*}
\| T_{\gamma}0\|_{\omega\textrm{-span}} &
=|\gamma|\, \| R_{\gamma}0\|_{\omega\textrm{-span}}
\leq |\gamma| M
\leq M,\\
\| T_{\gamma}^{2}0\|_{\omega\textrm{-span}}
& \leq \| T_{\gamma}^{2}0-T_{\gamma}0\|_{\beta,\omega\textrm{-span}}+\|T_{\gamma}0\|_{\beta,\omega\textrm{-span}} \leq
\|T_{\gamma}0\|_{\beta,\omega\textrm{-span}}(L+1)\\
& \leq \frac{|\gamma|}{1-L} \| R_{\gamma}0\|_{\beta,\omega\textrm{-span}}\leq \frac{|\gamma|}{\beta(1-L)} \|
R_{\gamma}0\|_{\omega\textrm{-span}}\leq M,\\
\| T_{\gamma}^{3}0\|_{\omega\textrm{-span}}
& \leq \| T_{\gamma}^{3}0-T^{2}_{\gamma}0\|_{\beta,\omega\textrm{-span}}+ \|
T_{\gamma}^{2}0-T_{\gamma}0\|_{\beta,\omega\textrm{-span}}+\|T_{\gamma}0\|_{\beta,\omega\textrm{-span}} \leq
\|T_{\gamma}0\|_{\beta,\omega\textrm{-span}}(L^{2}+L+1) \\
& \leq \frac{|\gamma|}{1-L} \| R_{\gamma}0\|_{\beta,\omega\textrm{-span}}\leq \frac{|\gamma|}{\beta(1-L)} \|
R_{\gamma}0\|_{\omega\textrm{-span}}\leq M,\\
\ldots & \leq \ldots\\
\| T_{\gamma}^{n}0\|_{\omega\textrm{-span}}
& \leq \|T_{\gamma}0\|_{\beta,\omega\textrm{-span}}(L^{n-1}+\ldots+L+1)\leq \frac{|\gamma|}{\beta(1-L)} \|
R_{\gamma}0\|_{\omega\textrm{-span}}\leq M.
\end{align*}

Using Banach's fixed point theorem (see e.g.~\cite[Appendix~A]{Her1989}), we know that there exists at most one fixed point of  $T_{\gamma}$ in
$\cC_{\omega}(\bR^{k})$ endowed with the $\omega$-span
norm. Exploiting the fact that $\| T^{n}_{\gamma}0\|_{\omega\textrm{-span}}\leq M$ for any $n\in\bN$ and the local contraction property of
$T_{\gamma}$ we conclude that there exists a unique $u_{\gamma}\in\bC_{\omega}(\bR^{k})$ (up to an additive constant), such that
\[\| T_{\gamma} u_{\gamma}- u_{\gamma}\|_{\beta,\omega\textrm{-span}}=0.
\]
Consequently, for a fixed $a\in\bR^{k}$, the constant $\lambda_{\gamma}:=\frac{T_{\gamma}u_{\gamma}(a)-u_{\gamma}(a)}{\gamma}$ and $u_{\gamma}\in\bC_{\omega}(\bR^{k})$ are solutions to Bellman equation~\eqref{eq:rsc:bellmaneq2}.

Thus, the constant $\lambda_{\gamma}:=R_{\gamma}v_{\gamma}(0)-v_{\gamma}(0)$ and $v_{\gamma}\in\bC_{\omega}(\bR^{k})$ are solutions to Bellman equation~\eqref{eq:rsc:bellmaneq}.
\end{proof}

In the end of this Section, let us show a corollary, which will be helpful later. To do so let us fix $a\in \bR^{k}$ and define $\bar{u}_\gamma(x):=u_\gamma(x)-u_\gamma(a)$ for $x\in \bR^{k}$.

\begin{corollary}\label{cont} Under the assumptions and notation of Proposition \ref{pr:RSC.Bellman.solution} the functions $(\gamma_0,0)\ni \gamma \mapsto \lambda_{\gamma}$ and  $(\gamma_0,0)\ni \gamma \mapsto \bar{u}_{\gamma}(x)$ for each $x\in \bR^{k}$  are continuous.
\end{corollary}
\begin{proof} Clearly when $u_\gamma$ is a solution to \eqref{eq:rsc:bellmaneq2} then $\bar{u}_\gamma$ is also a solution to \eqref{eq:rsc:bellmaneq2}.  By \eqref{eq:prop5.4.a} and the proof of Proposition \ref{pr:RSC.Bellman.solution} we have that $\|\bar{u}_\gamma\|_{\omega\textrm{-span}}\leq M$ and
\begin{equation}\label{estim}
|T_{\gamma}^m0(x)-\bar{u}_\gamma(x)-T_{\gamma}^m0(a)|\leq M (L(M))^m(2+\beta(M)\omega(x)+\beta(M)\omega(a))
\end{equation}
for any $x\in \bR^{k}$ and $\gamma$ from a compact subinterval of $(\gamma_0,0)$. By Proposition \ref{pr:RSC.feller} for each $m$ and fixed $x \in \bR^{k}$ the mappings $\gamma \to T_{\gamma}^m0(x)$
and $\gamma \to T_{\gamma}^m0(a)$ are continuous. Therefore when $\gamma_n \to \gamma<0$ we have, using \eqref{estim}, that
\begin{align}
|\bar{u}_{\gamma_n}(x)-\bar{u}_\gamma(x)| & \leq |T_{\gamma_n}^m0(x)-T_{\gamma}^m0(x)|+|T_{\gamma_n}^m0(a)-T_{\gamma}^m0(a)| \nonumber \\
& \phantom{=}+2 M (L(M))^m(2+\beta(M)\omega(x)+\beta(M)\omega(a))=a_{n,m}+b_{n,m}+c_m.
\end{align}
For a given $\epsilon$ we can choose $m$ such that $c_m\leq \epsilon$. Then letting $n\to \infty$ for fixed $m$ we obtain continuity of the mapping $\gamma \to \bar{u}_\gamma(x)$.
Following the proof of Proposition \ref{pr:RSC.feller} we can also show that the mapping $\gamma \to  T_{\gamma}\bar{u}_\gamma(x)$ is continuous.
Consequently, the mapping $\lambda \to \lambda_\gamma={T_\gamma \bar{u}_\gamma(x)-\bar{u}_\gamma(x)\over \gamma}$ is continuous, which completes the proof.
\end{proof}

\section{Optimal strategy}\label{S:Strategy}
It is straightforward to check, that under the assumptions and notation of Proposition \ref{pr:RSC.Bellman.solution}, we get that $v_{\gamma}(x)=\frac{u_{\gamma}(x)}{\gamma}$ and $\lambda_{\gamma}$ are solutions to Bellman equation \eqref{eq:rsc:bellmaneq}. Finally, we can link Bellman equation~\eqref{eq:rsc:bellmaneq} and \eqref{eq:rsc:bellmaneq2} to our initial problem~\eqref{eq:RSC:begin}.

\begin{proposition}\label{pr:RSC.Bellman.solution2}
Under \eqref{as:rsc:A.1}--\eqref{as:rsc:A.5}, there exists $\gamma_0<0$, such that for any $\gamma\in (\gamma_0,0)$, we get
\[
\lambda_{\gamma}\geq \sup_{H\in\cA}\varphi^{\gamma}(V^{H}) ,
\]
i.e. the optimal value in problem~\eqref{eq:RSC:begin} does not exceed the solution of Bellman equation \eqref{eq:rsc:bellmaneq}. Moreover, if $a_1$ in the assumption \eqref{as:rsc:A.4} is bounded from above, we have that the optimal value in~\eqref{eq:RSC:begin} is equal to $\lambda_{\gamma}$ and the optimal strategy is defined by selectors to the Bellman equation \eqref{eq:rsc:bellmaneq}.
\end{proposition}

\begin{proof}
This proof could be considered as a variation of the classical verification theorem from the theory of Risk Sensitive Control (see e.g. \cite[Theorem~2.1]{HerMar1996}).
Let $\gamma_0$ be given by \eqref{eq:gamma0} and for $\gamma\in (\gamma_0,0)$, let $u_{\gamma}$ and $\lambda_{\gamma}$ denote the solutions of Bellman equation \eqref{eq:rsc:bellmaneq2}.

First, we need to show that $\lambda_{\gamma}$ is an upper bound for any $\gamma\in (\gamma_0,0)$, i.e. that for any adapted strategy $H=(H_t)_{t\in\bT}$, we get
\begin{equation}\label{inequality}
\lambda_\gamma\geq\liminf_{t\to\infty}\frac{1}{t}\mu^{\gamma}\left(\sum_{i=0}^{t-1}F(X_t,H_t,W_t)\right).
\end{equation}
For $i\in\bT$ and $p>1$, such that $\gamma> p\gamma_0$, using \eqref{eq:rsc:bellmaneq2}, we have
\[
e^{u_{\gamma\over p}(X_{i})} \leq \bE[e^{u_{\gamma\over p}(X_{i+1})+ {\gamma\over p} F(X_{i},H_{i},W_{i})-\lambda_{\gamma\over p} {\gamma\over p}}|\cF_{i}].
\]
Consequently, using the tower property, we get
\[
e^{t\lambda_{\gamma\over p} {\gamma\over p}} \leq \bE[e^{u_{\gamma\over p}(X_{t})-u_{\gamma\over p}(X_0)+ {\gamma\over p} \sum_{i=0}^{t-1}F(X_i,H_i,W_i)}]
\]
for any $t\in\bT$. Equivalently, for $v_{\gamma}(x)=\frac{u_{\gamma}(x)}{\gamma}$, we get
\[
\lambda_{\gamma\over p} \geq \frac{1}{t}\mu^{\gamma\over p}\left(\sum_{i=0}^{t-1}F(X_i,H_i,W_i)+v_{\gamma \over p}(X_{t})-v_{\gamma \over p}(X_0)\right).
\]
It is hard to get rid of $v$ taking the limit, in the above inequality (note, for the case of bounded $v$ it is straightforward). Using Holder's inequality we know that for  $q=p/(p-1)$ we get
\[
\lambda_{\gamma \over p} \geq \frac{1}{t}\left[\mu^{\gamma}\left(\sum_{i=0}^{t-1}F(X_i,H_i,W_i)\right)+\mu^{q\gamma}\left(v_{\gamma \over p}(X_{t})-v_{\gamma \over p}(X_0)\right)\right]
\]
and consequently (for any $p>1$), since $v_{\gamma \over p}(X_{t})-v_{\gamma \over p}(X_0)\leq M (2+\omega(X_{t})+\omega(X_0))$ and $\lim_{t \to \infty} \frac{1}{t}\mu^{q\gamma}(\omega(X_t))=0$ we have
\[
\lambda_{\gamma\over p} \geq \liminf_{t\to\infty}  \frac{1}{t}\mu^{\gamma}\left(\sum_{i=0}^{t-1}F(X_i,H_i,W_i)\right).
\]
By continuity of $\gamma \to \lambda_\gamma$ (see Corollary \ref{cont} ), we have that $\lim_{p\to 1} \lambda_{\gamma \over p}=\lambda_\gamma$, which shows \eqref{inequality}.

Second, we show the optimality of the strategy defined by the Bellman equation \eqref{eq:rsc:bellmaneq}, when $a_1$ in \eqref{as:rsc:A.4} is bounded from above by $\tilde{a}$. Let us fix $\gamma\in (\gamma_0,0)$ and let $M>0$ be such that $\|v_\gamma\|_{\omega}\leq M$. For the strategy $\hat{H}$ determined by the Bellman equation \eqref{eq:rsc:bellmaneq2}, using monotonicity of $\mu_{\gamma}$, we get
\begin{align*}
\lambda_{\gamma} & =\frac{1}{t}\mu^{\gamma}\left(\sum_{i=0}^{t-1}F(X_i,\hat{H}_i,W_i)+v_{\gamma}(X_{t})-v_{\gamma}(X_0)\right)\\
& \leq \frac{1}{t}\mu^{\gamma}\left(\sum_{i=0}^{t-1}F(X_i,\hat{H}_i,W_i)+M(\omega(X_t)+1)-v_{\gamma}(X_0)\right)\\
& \leq \frac{1}{t}\mu^{\gamma}\left(\sum_{i=0}^{t-1}F(X_i,\hat{H}_i,W_i)+M\left(\sum_{i=1}^t b_1^{i-1}a_1(W_{t-i})+b_1^t\omega(X_0)+1\right)-v_{\gamma}(X_0)\right)\\
& \leq \frac{1}{t}\mu^{\gamma}\left(\sum_{i=0}^{t-1}F(X_i,\hat{H}_i,W_i)\right)+\frac{M\left(\frac{\tilde{a}}{1-b_1}+\omega(X_0)+1\right)-v_{\gamma}(X_0)}{t}.
\end{align*}
Letting $t\to \infty$ we obtain (taking into account \eqref{inequality})
\[\lambda_{\gamma}= \liminf_{t\to\infty}  \frac{1}{t}\mu^{\gamma}(\sum_{i=0}^{t-1}F(X_i,H_i,W_i)),\]
which completes the second part of the proof.
\end{proof}

\section{Exemplary dynamics}\label{S:RSC.Examples}
In this subsection let us present examples of dynamics for which assumptions  \eqref{as:rsc:A.1}--\eqref{as:rsc:A.5} are fulfilled.

\begin{example}\label{ex:rsc1} In this example, we shall set $\omega \equiv 0$ (equivalently, one might say that $\omega$ is bounded) and show
that our framework covers a wide class of dynamics in the classical case. The first example is taken from~\cite{Ste1999}. We will assume that
time $\bT=\bR_{+}$ is continuous, but we can only reshape our portfolio in
discrete time moments $n\in\bN$. For $n\in\bN$ and ($z=1,\ldots,k+m$), let us assume that $W^{z}_{n}$ denotes the trajectory of
$w_{z}(t)-w_{z}(n)$
($n\leq t\leq n+1$), where $\{w_{z}(t)\}_{z=1}^{k+m}$ are independent Brownian motions (which generate the filtration). Let us assume that the
dynamics of the risky assets and factors is given by
\begin{align*}
X_{n}^{j} & = b_{j}(X_{n-1})+\sum_{z=1}^{k+m}\delta_{jz}[w_{z}(n)-w_{z}(n-1)], & n\in\bN,\\
\frac{dS^{i}_{t}}{S^{i}_{t}} & =a_{i}(X_{n})\d t+\sum_{z=1}^{k+m}\sigma_{iz}\d w_{z}(t), & t\in [n,n+1),
\end{align*}
where for ($i=1,\ldots,m$), ($j=1,\ldots, k$) and ($z=1,\ldots, k+m$):  $a_i,b_i:\bR^{k}\to\bR$ are measurable and bounded functions, $b_i$ is
continuous, $\delta_{jz}\in\bR$, $\sigma_{iz}\in\bR$ and $\textrm{rank}((\sigma_{iz})_{z=1,\ldots,k+m})=k$. Let $h_{i}(t)$ denote the part of
the
capital invested at time $t$ in the $i$-th risky asset and let
\[U=\{(h_1,\ldots,h_m)\in [0,1]^{m}:\ \sum_{i=1}^{m}h_i=1\}.\]
Moreover, let $H^{i}_{n}=h_{i}(n)$. Using Ito's Lemma (see~\cite{Ste1999} for details) we get function $F$ of the form
\begin{align*}
F(X_n,H_n,W_n) & =\phantom{+} \sum_{i=1}^{m}\int_{n}^{n+1}a_i(X_n)h_i(s)\d
s-\frac{1}{2}\sum_{z=1}^{k+m}\int_{n}^{n+1}\Big(\sum_{i=1}^{m}h_{i}(s)\sigma_{iz}\Big)^{2}\d s\\
&\phantom{=} +\int_{n}^{n+1}\sum_{i=1}^{m}h_{i}(s)\sum_{z=1}^{k+m}\sigma_{iz}\d w_{z}(s).
\end{align*}
One can check that assumptions \eqref{as:rsc:A.1}--\eqref{as:rsc:A.4} will hold in this framework, for $\omega \equiv 0$. See \cite{Ste1999},
where in fact equivalents of
all Propositions from Section~\ref{S:RSC.Bellman} are directly proved. For clarity, let us show the existence of the upper bound in
\eqref{as:rsc:A.4}, for function $F$. We get
\begin{align*}
F(X_{n},H_{n},W_{n}) & =\ln\frac{V_{n+1}}{V_{n}}=\ln\sum_{i=1}^{m}H^{i}_{n}\frac{S^{i}_{n+1}}{S^{i}_{n}}=\ln  \sum_{i=1}^{m} H_n^i
e^{a_i(X_n)+\sum_{z=1}^{k+m} \sigma_{iz}[w_z(n+1)-w_z(n)]}\\
& \leq  \sup_{1\leq i\leq m}\big(a_i(X_n)+\sum_{z=1}^{k+m} \sigma_{iz}[w_z(n+1)-w_z(n)]\big)\\
& \leq \| a\|_{\textrm{sup}}+\|\sigma\|_{\textrm{sup}}\max_{1\leq z\leq k+m}[w_z(n+1)-w_z(n)],
\end{align*}
where $\|a\|_{\sup}=\sup_{1\leq i\leq m}\sup_{x\in\bR^{k}}|a_{i}(x)|$ and $\|\sigma\|_{\sup}=\sup_{1\leq i\leq m}\sup_{1\leq z\leq
k+m}|\sigma_{iz}|$.

Thus, is is sufficient to set any $b_{2}\geq 0$ and
\[
a_{2}(w)= \| a\|_{\textrm{sup}}+\|\sigma\|_{\textrm{sup}}\max_{1\leq z\leq k+m}|w_z(n+1)-w_z(n)|(w).
\]
Note, it is easy to check that $a_2$ will satisfy \eqref{eq:assumpt.kpm}, as for a Gaussian $X$, we get $e^{|X|}\in
L^1$. Moreover \eqref{eq:assumpt.rsc.G} follows from boundedness of $b$ while \eqref{eq:rsc:unerg} from nondegeneracy of $\sigma$ and boundedness of $b$ and in fact one can find a constant $c$ uniform for all $x\in \bR^{k}$.
In this example a solution to the Bellman equation \eqref{eq:rsc:bellmaneq} is bounded and therefore we obtain in Proposition \ref{pr:RSC.Bellman.solution2} that $\lambda_\gamma$ is the optimal value without additional assumptions.
\end{example}
\begin{example} We shall now generalize previous example. Namely, let
\[G(x,W)=B(x)+C(W),\]
where $B:\bR^{k} \to \bR^{k}$ is such that $\|B(x)\|\leq A+b_1\|x\|$ with $b_1<1$ and $C: \bR^{k+m}\to \bR^{k}$ is bounded from above of the form
\[
C(W_n)=\min\left\{\sum_{z=1}^{k+m}\delta_{jz}[w_{z}(n)-w_{z}(n-1)]\, ,\, K\right\},
\]
with $K>0$. Then
\begin{align*}
X_n&=B(X_{n-1})+C(W_n), & \\
\frac{dS^{i}_{t}}{S^{i}_{t}} & =a_{i}(X_{n})\d t+\sum_{z=1}^{k+m}\sigma_{iz}\d w_{z}(t), & t\in [n,n+1),
\end{align*}
where we assume that $\|a_i\|_{\omega}<\infty$. Choosing $\omega(x)=a+b_1\|x\|$ one can check that all assumptions \eqref{as:rsc:A.1}--\eqref{as:rsc:A.5} together with boundedness from above of $a_1$ in \eqref{as:rsc:A.4} are satisfied. In particular, assumption \eqref{as:rsc:A.5} is satisfied uniformly in $x\in \bR^k$ from compact sets due to the form of $G(x,W)$ and $C(W_n)$.
\end{example}

\begin{example} Let us assume that assumptions \eqref{as:rsc:A.1} and \eqref{as:rsc:A.2} hold and the dynamics of $i$-th risky assets is given by
\[
\frac{S^{i}_{t+1}}{S^{i}_{t}}=\xi_{i}(X_t,W_t),
\]
for any $t\in\bT$, where $\xi_{i}$ is a measurable vector function. Moreover the set $U$ will be of the form
$\{(h_1,\ldots,h_m)\in [0,1]^{m}:\ \sum_{i=1}^{m}h_i\leq1\}$.
Then we can define $F$ explicitly, as
\[
F(X_n,H_n,W_n)=\ln\left(\sum_{i=1}^{m}H^{i}_{n}\xi_{i}(X_{n},W_n)+(1-\sum_{i=1}^{m}H^{i}_{n})\right).
\]
To get assumptions \eqref{as:rsc:A.3} and \eqref{as:rsc:A.4} we need to impose additional assumptions on $W$ and $\xi_{i}$. In particular we can
consider the discretized version of Example~\ref{ex:rsc1} by setting $W_n^{i}=w_{i}(n+1)-w_{i}(n)$ and
\begin{equation}\label{eq:rsc:ex2}
\xi^{i}(X_n,W_n)=\exp \Big\{a_{i}(X_n)-\frac{1}{2}\sum_{z=1}^{k+m}\sigma^{2}_{iz}+\sum_{z=1}^{k+m}\sigma_{iz}W^{j}_{n}\Big\}.\end{equation}
See \cite{Ste2004b} for details in general case and \cite{DiMSte2006b} for the case when \eqref{eq:rsc:ex2} holds.

\end{example}

{\small
\bibliographystyle{amsplain}
\providecommand{\bysame}{\leavevmode\hbox to3em{\hrulefill}\thinspace}
\providecommand{\MR}{\relax\ifhmode\unskip\space\fi MR }
\providecommand{\MRhref}[2]{%
  \href{http://www.ams.org/mathscinet-getitem?mr=#1}{#2}
}
\providecommand{\href}[2]{#2}

}

 \end{document}